\documentclass[11pt]{article}

\usepackage{amssymb,latexsym,todonotes}

\usepackage{graphicx}
\usepackage{amsmath}
\usepackage{bbm}
\usepackage{cite}
\usepackage[sort&compress,numbers]{natbib}
\usepackage{amsfonts}
\usepackage{amssymb}
\usepackage{amsmath}
\usepackage{latexsym}
\usepackage{color}
\usepackage{url}
\usepackage{ntheorem}
\usepackage{braket}
\usepackage[normalem]{ulem}
\usepackage{tikz}
\usepackage{enumerate}

\newcommand{\sqi}{\red{\mathrm{i}}}

\newcommand{\fourgh}{\red{\mathbf{g}}}

\newcommand\ben{\begin{enumerate}}
\newcommand\een{\end{enumerate}}
\newcommand\bit{\begin{itemize}}
\newcommand\eit{\end{itemize}}

\newcommand{\blue}[1]{{\color{blue}#1}}
\newcommand{\red}[1]{{\color{red}#1}}

\newcounter{mnotecount}[section]

\renewcommand{\themnotecount}{\thesection.\arabic{mnotecount}}

\newcommand{\mnote}[1]
{\protect{\stepcounter{mnotecount}}$^{\mbox{\footnotesize
$
\bullet$\themnotecount}}$ \marginpar{
\raggedright\tiny\em
$\!\!\!\!\!\!\,\bullet$\themnotecount: #1} }

\newcommand{\RR}{\R}

\newtheorem{theorem}{\sc  Theorem\rm}[section]

\newtheorem{Theorem}[theorem]{\sc  Theorem\rm}
\newtheorem{thm}[theorem]{\sc  Theorem\rm}

\newtheorem{cor}[theorem]{\sc  Corollary\rm}
\newtheorem{corollary}[theorem]{\sc  Corollary\rm}

\newtheorem{lemma}[theorem]{\sc Lemma\rm}

\newtheorem{prop}[theorem]{\sc Proposition\rm}
\newtheorem{proposition}[theorem]{\sc Proposition\rm}

\newtheorem{rem}[theorem]{\sc Remark\rm}
\newtheorem{remark}[theorem]{\sc Remark\rm}

\newcommand{\jlcax}[1]{}
\newcommand{\eean}{\nonumber\end{eqnarray}}

\newcommand{\ii}{\mathrm{i}}

\newcommand{\kk}[1]{}

\newcommand{\cS}{{\cal S}}

\newcommand{\beq}{\begin{equation}}

\newcommand{\ep}{\epsilon}

\newcommand{\Rn}{\red{\,{}^{n}\!R}}
\newcommand{\FS}       
                  {F}

\newcommand{\HS} 
       {H_{\mbox{\scriptsize volume}}}

\newcommand{\half}{\frac 12}

{\ptc{this should be removed in the oberwolfach version}}%

\newcommand{\eeal}[1]{\label{#1}\end{eqnarray}}
\newcommand{\bed}{\begin{deqarr}}
\newcommand{\eed}{\end{deqarr}}
\newcommand{\bedl}[1]{\begin{deqarr}\label{#1}}
\newcommand{\eedl}[2]{\arrlabel{#1}\label{#2}\end{deqarr}}

\newcommand{\bel}[1]{\begin{equation}\label{#1}}
\newcommand{\bea}{\begin{eqnarray}}
\newcommand{\bean}{\begin{eqnarray}\nonumber}
\newcommand{\beal}[1]{\begin{eqnarray}\label{#1}}
\newcommand{\eea}{\end{eqnarray}}

\newcommand{\dist}{\mathrm{dist}}

\newcommand{\nn}{\nonumber}
\newcommand{\Eq}[1]{Equation~\eq{#1}}

\newcommand{\qed}{\hfill $\Box$ \medskip}

\newcommand{\be}{\begin{equation}}
\newcommand{\eeq}{\end{equation}}
\newcommand{\ee}{\end{equation}}
\newcommand{\beqa}{\begin{eqnarray}}
\newcommand{\eeqa}{\end{eqnarray}}
\newcommand{\beqan}{\begin{eqnarray*}}
\newcommand{\eeqan}{\end{eqnarray*}}
\newcommand{\ba}{\begin{array}}
\newcommand{\ea}{\end{array}}

\newcommand{\Id}{\mbox{\rm Id}} 

\newcommand{\warn}[1]
{\protect{\stepcounter{mnotecount}}$^{\mbox{\footnotesize
$
\bullet$\themnotecount}}$ \marginpar{
\raggedright\tiny\em
$\!\!\!\!\!\!\,\bullet$\themnotecount: {\bf Warning:} #1} }

\newcommand{\R}{\mathbb R}

\newcommand{\eq}[1]{(\ref{#1})}

\newcommand{\Mext}{M_\ext}
\newcommand{\ext}{\mathrm{ext}}

\newcommand{\ptc}[1]{\mnote{{\bf ptc:}#1}}

\newcommand{\Ric}{\mbox{\rm Ric}}

\newcommand{\mcL}{{\mycal L}}

\newcommand{\beqar}{\begin{deqarr}}
\newcommand{\eeqar}{\end{deqarr}}

\newcommand{\beaa}{\begin{eqnarray*}}
\newcommand{\eeaa}{\end{eqnarray*}}

\newcommand{\supp}{{\mbox{\rm supp}\,}}

\renewcommand{\Rn}{{\R^n}}

\newcommand\bXb{\partial \Xb} 
\newcommand\Xb{\red{\bar M}}
\newcommand{\Xbext}{\red{M_{\textrm{\rm ext}}}} 
\newcommand{\Xext}{\blue{M_{\textrm{\rm ext}}}} 

\newcommand{\rdelta}{\red{\delta}}

\newcommand{\tdelta}{{\tilde\delta}}

\newcommand\deltat{\tdelta}

\newcommand{\rmnote}[1]{}

\newcommand\Mwhere{\ \text{where}\ }

\newcommand\Real{\mathbb{R}}

\newcommand\NN{\mathbb{N}}

\newcommand\im{\operatorname{Im}}
\newcommand\re{\operatorname{Re}}

\newcommand\pa{\partial}

\newcommand\sphere{\mathbb{S}}
\newcommand\Sn{\sphere^{n-1}}

\newcommand\CI{{\mathcal C}^{\infty}}
\newcommand\Cinf{{\mathcal C}^{\infty}}
\newcommand\dCinf{{\dot{\mathcal C}}^{\infty}}
\newcommand\dCI{{\dot{\mathcal C}}^{\infty}}

\renewcommand{\Box}{{\square}}

\newcommand\Diff{\operatorname{Diff}}
\newcommand\Diffb{\operatorname{Diff}_{\text{b}}}
\newcommand\Diffsc{\operatorname{Diff}_{\text{sc}}}
\newcommand\Diffsch{\operatorname{Diff}_{\text{sc},h}}

\newcommand\Hsc{H_{\text{sc}}}
\newcommand\Hsch{H_{\text{sc},h}}

\newcommand\Psop{\operatorname{\Psi}}

\newcommand\psit{\tilde\psi}

\newcommand\bl{{\text b}}
\newcommand\scl{{\text{sc}}}
\newcommand\sccl{{\text{scc}}}

\newcommand\Tsc{{}^\scl T}
\newcommand\Diffscc{\Diff_\sccl}
\newcommand\Diffscch{\operatorname{Diff}_{\text{scc},h}}

\newcommand\Psisc{\Psop_\scl}

\newcommand\Psiscc{\Psop_\sccl}
\newcommand\Psiscch{\Psop_{\sccl,h}}

\newcommand\Vb{{\mathcal V}_{\bl}}
\newcommand\Vsc{{\mathcal V}_{\scl}}
\newcommand\Vsch{{\mathcal V}_{\scl,h}}

\numberwithin{equation}{section}
\theoremstyle{remark}
\theoremstyle{definition}

\DeclareFontFamily{OT1}{rsfs}{}
\DeclareFontShape{OT1}{rsfs}{m}{n}{ <-7> rsfs5 <7-10> rsfs7 <10-> rsfs10}{}
\DeclareMathAlphabet{\mycal}{OT1}{rsfs}{m}{n}

{\catcode `\@=11 \global\let\AddToReset=\@addtoreset}
\AddToReset{equation}{subsection}

{\catcode `\@=11 \global\let\AddToReset=\@addtoreset}
\AddToReset{figure}{section}

{\catcode `\@=11 \global\let\AddToReset=\@addtoreset}
\AddToReset{table}{section}

\presetkeys{todonotes}{fancyline, color=green!40, size=\tiny}{}

\renewcommand{\red}[1]{#1}
\renewcommand{\blue}[1]{#1}
\begin{document}

\title{Asymptotically flat Einstein-Maxwell fields are inheriting\protect\thanks{Preprint UWThPh-2017-43}}
\author{Piotr T. Chru\'{s}ciel\thanks{Faculty of Physics and Erwin Schr\"odinger Institute, Vienna.
{\sc Email} \protect\url{piotr.chrusciel@univie.ac.at}, {\sc URL} \protect\url{homepage.univie.ac.at/piotr.chrusciel}},
 Luc
Nguyen\thanks{University of Oxford and Erwin Schr\"odinger Institute, Vienna},
Paul Tod$^\ddag$
and
Andr\'as Vasy\thanks{Stanford University and Erwin Schr\"odinger Institute, Vienna}}
\maketitle

\begin{abstract}
We prove that  Maxwell fields of asymptotically flat solutions of the Einstein-Maxwell equations inherit the
stationarity of the metric.
\end{abstract}


\tableofcontents

\section{Introduction}
\label{sI}

In the study of static or stationary Einstein-Maxwell solutions of
Einstein's equations, including black hole solutions, it is frequently assumed that the Maxwell field is
also static or stationary, in the sense that the Lie derivative of the
Maxwell tensor $F_{\red{ij}}$ along the time-translation Killing vector $K^a$ is zero. Evidently one needs the
energy-momentum tensor $T_{\red{ij}}$ to be static or stationary in this sense for the Einstein
equations to be consistent, but this does not actually require that $F_{\red{ij}}$ be
static or stationary -- there could be a duality rotation of the electromagnetic field as one moves along the Killing vector, a transformation that is known to leave the energy-momentum tensor invariant.
It is customary (see e.g. \cite{MV}) to say that $F_{\red{ij}}$ does not {\emph{inherit}} the
symmetry if $T_{\red{ij}}$ is static or stationary but $F_{\red{ij}}$ is not, and then one can consider non-inheriting solutions.
Such solutions are well known, we review them in Section~\ref{s19IX17.1} below.

The question was raised in \cite{t1} whether non-inheriting, static or stationary Einstein-Maxwell solutions could be asymptotically-flat.\footnote{Something close to this was also asked in \cite{mhrs}.} Arguments were
presented that under stronger conditions (analyticity up to and including the horizon)
there could be no strictly non-inheriting static Einstein-Maxwell black holes but it was left open whether this undesirable analyticity requirement could
be dispensed with. (It is known that analyticity holds away from the horizon,
as in the inheriting case \cite{t2}.) We return to the question in this paper and show that
the Maxwell fields of
asymptotically-flat  solutions of Einstein-Maxwell equations inherit the stationarity of the metric.
Hence, neither strictly stationary (``soliton'') solutions
nor black hole solutions which are asymptotically flat and non-inheriting exist.
In fact, the mere existence of an asymptotically flat end on a spacelike hypersurface suffices to prove inheritance, without any further global conditions.

This paper is organised as follows: In Section~\ref{s19IX17.1} we review some non-inheriting solutions. In Section~\ref{s19IX17.3} we outline the derivation of the equations at hand. In Section~\ref{s19IX17.2} we show that the metrically-static solutions must be inheriting
using  integration-by-parts   arguments.
 In Section~\ref{s19IX17.4}
we present a proof  which covers the metrically-strictly-stationary solutions, adapting and extending the arguments in~\cite{Vasy:Propagation-2}. While the arguments in Section~\ref{s19IX17.2} are hand-tailored for the problem at hand, the ones in Section~\ref{s19IX17.4} provide a general  result  which applies to a wide class of equations.

It is conceivable that the results of~\cite{t1,t2} together with the unique continuation results of \cite{Mazzeo-continuation} can be used to exclude inheriting solutions with bifurcate Killing horizons, independently of asymptotic conditions, but we have not explored further this line of thought.

\section{Non-inheriting solutions}
 \label{s19IX17.1}

Some historical comments about the problem at hand are in order.
Noninheritance of symmetry is discussed in \cite[Section 11.1]{skmhh}. Another good reference in the static case is \cite{MV}. These authors reference earlier literature and reduce the static,
cylindrically-symmetric Einstein-Maxwell equations, when the staticity is not inherited by the Maxwell field,
to a set of coupled ordinary differential equations. They conclude that such solutions therefore exist, at least locally, contradicting a conjecture then current.
Among their references is \cite[Equation~(5.2)]{mc}, with the  solution metric
\[
 \fourgh =-(dt-br^2d\phi)^2
  + e^{b^2r^2}(dz^2+dr^2)
  +r^2d\phi^2
 \,,
\]
with real constant $b$, while the Maxwell potential is
\[
 A=\cos(2bz)(dt-br^2d\phi)
 \,.
\]
The metric has Killing vectors $\partial/\partial t, \partial/\partial\phi$ and $\partial/\partial z$ and the last of these is not inherited by the Maxwell potential. The Maxwell field undergoes a duality
rotation when Lie-dragged along $\partial/\partial z$, so that in this case ${\blue{a}}=-2b$. The metric is not asymptotically-flat, and in fact it is not orthogonally-transitive with respective to
either of the two isometry groups generated
respectively by $\langle \partial/\partial t, \partial/\partial z \rangle$ or $\langle \partial/\partial\phi, \partial/\partial z\rangle $.

A simple explicit solution given in \cite[equation (24.46)]{skmhh}, and discussed in \cite{t2}, which illustrates some other possibilities identified in \cite{t1}, is provided by the (conformally-flat) plane-wave metric
\[
 \fourgh =- 2du(dv+b^2\zeta\overline{\zeta}du)
 +2d\zeta d\overline{\zeta}
 \,,
\]
with real constant $b$ and Maxwell field
\[F=be^{-\ii f(u)}du\wedge d\overline{\zeta}+c.c.
 \,,
\]
with arbitrary real $f(u)$ (which does not appear in the metric). Here the bar denotes complex conjugation. This is Einstein-Maxwell for any $f(u)$, and the (time-like or null) Killing vector $K=\partial/\partial u$ is not inherited: one has
\[{\mathcal{L}}_KF=-{\blue{a}} F^\star\mbox{  with  }{\blue{a}}=f'(u),
\]
where $F^\star$ is the dual of $F$.
(There are several more symmetries of this metric.) In this example the Maxwell field is null, as it has to be if ${\blue{a}}$ is to be non-constant, and the metric also admits a
twist-free, shear-free null geodesic congruence, here tangent to $\partial/\partial v$, which again is necessary for non-constant ${\blue{a}}$ -- see \cite{t1} for a proof. Choosing $f(u)$ non-analytic, one obtains
an example of a stationary but non-analytic
solution of the Einstein-Maxwell equations.

\section{The equations}
 \label{s19IX17.3}

For completeness we review the derivation of the key equations in both static and stationary cases.

\subsection{Static case}

Following~\cite{t1}, we assume a static
metric, which we write in the form
\begin{equation}
\label{met1}
\fourgh
 =-V^2dt^2 +\red{g}_{ij}(x^k)dx^idx^j
  \,,
\end{equation}
where the Killing vector is $K=\partial/\partial t$, thus
$\fourgh(K,K)=-V^2$. The non-inheriting conditions on the Maxwell field tensor $F_{\red{ij}}$ and its dual $F^\star_{\red{ij}}$
 are
\begin{equation}\label{inh1}\mathcal{L}_KF_{\red{ij}}=-{\blue{a}} F^\star_{\red{ij}},\;\;\mathcal{L}_KF^\star_{\red{ij}}={\blue{a}} F_{\red{ij}} 
\,.
\end{equation}
In the static case one needs
\[T_{\red{ij}}K^j=fK_i
 \,,
\]
for some real, non-negative $f$, where $T_{\red{ij}}$ is the Maxwell stress tensor, since the momentum constraint requires $T_{0i}=0$ which implies this.
This in turn prevents $F_{\red{ij}}$ from being null except where it vanishes (if it ever vanishes).
Now the source-free Maxwell equations impose, in form language:
\[d{\blue{a}}\wedge F=0=d{\blue{a}}\wedge F^\star,\]
and non-nullness of $F$ imposes $d{\blue{a}}=0$ and so ${\blue{a}}$ is a real, nonzero constant.
%
In the region where $V>0$ the electric and magnetic field vectors are defined by
\[
 E_i=V^{-1}K^jF_{\red{ji}}\;;\;\;\;B_i=V^{-1}K^jF^\star_{\red{ji}}.
\]
It was shown in (\ref{inh1}) that $F_{\red{ij}}$ vanishes at the bifurcation surface of a bifurcate horizon, so that $E_{i}, B_{i}$ end up finite on such a horizon, but in any case this
will not be an issue since our proof proceeds by unique continuation from infinity.

Then non-inheriting can be shown to imply
\begin{equation}
\label{w1}
E_{i}=W_{i}\sin({\blue{a}} t)\,,\;\;B_{i}=-W_{i}\cos({\blue{a}} t)\,,
\end{equation}
for some real $W_{i}$ orthogonal to $K^a$, and ${\blue{a}}$ as before. The Einstein-Maxwell equations become
\begin{eqnarray}
 -\Delta V&=&\frac{1}{2}(\red{g}^{ij}W_iW_j)V,\label{v1}
\\
 \epsilon_i^{\;\;jk}\red{\nabla}_j(VW_k)&=&-{\blue{a}} W_i,\label{w4}
\\
R_{ij}-\frac{1}{2}R\red{g}_{ij}&=&V^{-1}\red{\nabla}_i\red{\nabla}_jV-W_iW_j
 \,,
  \label{r1}
\end{eqnarray}
where these are all 3-dimensional quantities:   $R_{ij},R$ are 3-dim Ricci tensor and scalar, $\red{\nabla}_i$ is 3-dim Levi-Civita covariant derivative, $\Delta=-\red{g}^{ij}\red{\nabla}_i\red{\nabla}_j$ etc.

Simplify (\ref{w4}) by redefining and rescaling:
\[\omega_i:=VW_i,\;\;\red{\hat{g}}_{ij}:=V^{-2}\red{g}_{ij},\]
for then
\[
 \hat{\epsilon}_{ijk}=V^{-3}\epsilon_{ijk},\;\;\hat{\epsilon}_i^{\;\;jk}=V\epsilon_i^{\;\;jk}
 \,,
\]
and (\ref{w4}) becomes
\[\hat{\epsilon}_i^{\;\;jk}\partial_j\omega_k=-{\blue{a}}\omega_i\]
indices raised by $\red{\hat{g}}$, and in form language this is
\begin{equation}\label{4I18.1}
*d\omega+{\blue{a}}\omega=0
 \,,
\end{equation}
where $*$ is the 3-dim Hodge dual. Trace (\ref{r1}) and use (\ref{v1}) to obtain
\[
 R=\frac12\red{g}^{ij}W_iW_j,
\]
The hypothesis that the metric is asymptotically flat and has a well defined and finite energy leads to the condition
\begin{equation}\label{19OX17.1}
  W \in L^2(\Mext)
   \,.
\end{equation}
where $\Mext$ denotes the asymptotically flat region $\{|x|\ge R\}$,
which  translates to the same requirement for $\omega$.

\subsection{Stationary case}

We mostly follow \cite{t2}, but with some different conventions. We still have (\ref{inh1}), but ${\blue{a}}$ does not have to be constant anymore: as we saw in the preceding section, non-trivial (and not asymptotically flat) solutions with a non-constant lambda have been found. From the method of proof of Theorem 1.3 of~\cite{t1} it also follows that:

\begin{theorem}
  \label{T25IX17.1}
Let $(M,
\fourgh,F)$ be an Einstein-Maxwell space-time with a  Killing vector field $K_{i}$ such that
$\mcL_KF_{\red{ij}} = {\blue{a}} F^\star_{\red{ij}}$. If ${\blue{a}}$ is not constant  then $F$ is a null Maxwell field and the space-time admits
a non-twisting, shear-free null geodesic congruence.
\end{theorem}

The metric as in Theorem~\ref{T25IX17.1} lies in the Robinson-Trautman class if the expansion is non-zero,
or the Kundt class if the expansion is zero, or is a pp-wave   if the generator of
the congruence can be chosen to be parallel. Such solutions are unlikely to be asymptotically flat, and therefore will not be considered any longer here. In fact, it is known that pp-waves cannot be asymptotically flat by~\cite{ChBeig1} (see also~\cite{HuangLee}),
but it should be admitted that the remaining cases are not entirely clear (compare~\cite{CTMason,Mason:AlgSpec}).
For the sake of completeness it would be of interest to settle this.

In our analysis below we will assume that ${\blue{a}}$ is a nonzero, real constant
and take the metric to be
\[
 \fourgh=-V^2(dt+\theta_idx^i)^2+\red{g}_{ij}dx^idx^j
 \,.
\]
Introduce $E_{i}$ and $B_{i}$ as before (these are not now closed).
We find
\[
 E_{i}+\red{\mathrm{i}}B_{i}=V^{-1}\zeta_{i}e^{\red{\mathrm{i}}{\blue{a}} t}
 \,,
\]
where $\zeta_{i}$, the counterpart of $\omega_{i}$   considered in the static case, is now complex but still with ${\mathcal{L}}_K\zeta_{i}=0$.
The relevant equation (unnumbered in \cite{t2} but just before (32) there) turns out to be
\begin{equation}\label{21XII17.11}
 \epsilon_i^{\;\;jk}(\partial_j\zeta_k-\red{\mathrm{i}}{\blue{a}} \theta_j\zeta_k)=-{\blue{a}} V^{-1}\zeta_i
  \,,
\end{equation}
or in form notation
\beal{26IX17.11}
 &
 *(d\zeta-\red{\mathrm{i}}{\blue{a}}\theta\wedge\zeta)=-{\blue{a}} V^{-1}\zeta
 \,.
 &
\eea
This is similar in form to the static case (the $V^{-1}$ factor on the right-hand side can be absorbed into a conformal rescaling of the spatial metric) but with the exterior derivative `twisted' by $\theta_i$.
 
\Eq{21XII17.11} implies, away from the zeros of $V$, 
\beal{26IX17.12}
\nabla_i\zeta^i
  & =
   &
   \big((\ln V)_{,i}+\sqi V\,\epsilon_i^{\;\;jk}\partial_j\theta_k   -\sqi {\blue{a}}\theta_i\big)\zeta^i
 \,,
\\
 (-\Delta  + {\blue{a}}^2 )\zeta_\ell
 & = & {\blue{a}}^2 (1-V^{-2}) \zeta_\ell
  +
  {\blue{a}} \epsilon_\ell{}^{ij}(\nabla_i (V^{-1})
   +
    \sqi \theta_i) \zeta_j
 \nn
\\
 &&   + \nabla_\ell
   \left(\sqi \epsilon^{ijk} \nabla_i (\theta_j \zeta_k) - \zeta^i \nabla_i (V^{-1})
    \right)
    \nn
\\
 &&
     - \sqi {\blue{a}} \nabla^k(\theta_\ell \zeta_k - \theta_k \zeta_\ell)
   \,.
  \label{26IX17.13}
\eea
Note that the second derivatives of $\zeta$ appearing at the right-hand side of the last equation can be replaced by lower-order ones using \eqref{21XII17.11}. This provides a homogeneous second-order equation for $\zeta$ with diagonal principal part to which our uniqueness analysis of Section~\ref{s19IX17.4} applies, leading to the vanishing of $\zeta$ for field configurations satisfying the asymptotic flatness conditions.

\section{Static case}
 \label{s19IX17.2}

We return to \eqref{4I18.1}, where the lapse function $V$ has been  absorbed into a redefinition of the metric.
Henceforth we consider a complete three
 dimensional Riemannian manifold $(M^3,g)$
  with or without boundary and with a one-form satisfying
\begin{equation}
*d\omega = a\,\omega
	\;,
	\label{Eq:20VIII17-1}
\end{equation}
where $a \in \RR \setminus\{0\}$.
Note that \eqref{Eq:20VIII17-1} implies that
\begin{equation}
\delta\omega = 0
	\;.
	\label{Eq:20VIII17-1X}
\end{equation}

From \eqref{Eq:20VIII17-1} and \eqref{Eq:20VIII17-1X}, the Hodge Laplacian of $\omega$ is found to be
\begin{equation}
\Delta_H \omega = \delta d \omega + d \delta \omega = * d (* d\omega) = a^2 \omega
	\;.\label{Eq:08IX17-1}
\end{equation}
Thus, by Weitzenbock formula,
\begin{equation}
\nabla^i \nabla_i \omega_j  = - \Delta_H \omega_j + R_{ij} \omega^i= -  a^2 \omega_j + R_{ij} \omega^i
	\;.\label{Eq:08IX17-2}
\end{equation}

We will say that $(M,g)$ contains an asymptotically flat end $\Mext=\{|x|\ge R\}$  for some $R$,
 if
 it holds that
\begin{equation}
|g_{ij}(x) - \delta_{ij}| + |x||\partial g_{ij}(x)| + |x|^2|\partial^2 g_{ij}(x)| \leq C_*\,|x|^{-\delta }
	 \label{Eq:09IX17-A1}
\end{equation}
for some $C_* > 0$ and $\delta  > 0$, and where $|x|$ denotes the Euclidean norm.

In the remainder of this section we will prove:

\begin{Theorem}
  \label{T19IX17.1}
   If $(M,g)$ contains an asymptotically flat end and $\omega \in L^2(\Mext)$
    then $\omega \equiv 0$.
\end{Theorem}

In fact, we will prove a slightly stronger statement, see Proposition \ref{Prop:19I18-P1}. The proof is a unique continuation argument which uses a suitable monotone Almgren-type frequency function; see \cite{Almgren-1979, GarofaloLin-1987}, compare \cite{Bernstein-2017} for a recent application of this method to geometric problems.
See also~\cite[Sections~9-10]{RodnianskiTaoLAP} for a systematic treatment of related integral quantities and ODEs in the context of the closely related limiting absorption principle.

In a nutshell, our argument examines the decay rate of the normalized $L^2$-norm of $\omega$ on large `spheres' (i.e. the quantity $X(r)$ defined below in \eqref{Eq:XDef}, where $r$ is a suitably defined distance function). The Almgren-type frequency $F(r)$ is then recognized as a perturbation of $\frac{- r  X'(r)}{2X(r)}$. We will show that $F$ is non-increasing, which implies that the decay rate of $X$ is not faster than  polynomial unless $\omega \equiv 0$. A more quantitative estimate for the derivative of $F$ shows in fact that $X$ cannot decay faster than $r^{-3}$, unless $\omega \equiv 0$. But then the same estimate together with the assumption that $\omega \in L^2$, implies that $X$ must decay at least as fast as $r^{-3}$, which concludes our argument.

\subsection{Preliminaries about distance functions}

We collect here some facts about the distance function which we will use later on. For large $R \gg 1$, let $B_R(0)$ denote the coordinate ball of radius $R$ and define
\[
d_R(x) := \dist(x, B_R(0))
	\;.
\]
We first show that, for large $R$, $d_R$ is smooth outside of $B_R(0)$. We will use the following weighted Poincar\'e inequality, which we prove for completeness:

\begin{lemma}\label{Lem:WPIneql}
Fix $\ell > 0$ and $\tau > 0$. For any $\varphi \in C^1([0,\ell])$ with $\varphi(\ell) = 0$, there holds
\[
\int_{0}^\ell \frac{\varphi^2(t)}{(R + t)^{2 + {\red{\delta}} }}\,dt \leq \frac{4}{(1 + {\red{\delta}} )^2R^{2{\red{\delta}} }} \int_{0}^\ell |\varphi'(t)|^2\,dt
	\;.
\]
\end{lemma}

\begin{proof} We compute
\begin{align*}
\int_{0}^\ell \frac{\varphi^2(t)}{(R + t)^{2 + {\red{\delta}} }}\,dt
	&= - \int_0^\ell \frac{2}{(R + t)^{2+{\red{\delta}} }} \int_t^\ell \varphi(s)\,\varphi'(s)\,ds\,dt\\
	&=  \frac{2}{(1 + {\red{\delta}} )R^{1+{\red{\delta}} }} \int_0^\ell \varphi(s)\,\varphi'(s)\,ds
		+ \int_0^\ell \frac{2\varphi(t)\,\varphi'(t)}{(1 + {\red{\delta}} )(R + t)^{1+{\red{\delta}} }}\,dt\\
	&=  -\frac{ \varphi^2(0)}{(1 + {\red{\delta}} )R^{1+{\red{\delta}} }}
		+ \int_0^\ell \frac{2\varphi(t)\,\varphi'(t)}{(1 + {\red{\delta}} )(R + t)^{1+{\red{\delta}} }}\,dt\\
	&\leq \int_0^\ell \frac{2\varphi(t)\,\varphi'(t)}{(1 + {\red{\delta}} )(R + t)^{1+{\red{\delta}} }}\,dt\\
	&\leq \frac{2}{(1 + {\red{\delta}} )R^{\red{\delta}} } \Big(\int_0^\ell |\varphi'(t)|^2\,dt\Big)^{1/2}\Big(\int_0^\ell \frac{\varphi^2(t)}{(R + t)^{2+{\red{\delta}} }}\,dt\Big)^{1/2}
		\;.
\end{align*}
The conclusion is readily seen.
\qed
\end{proof}

\begin{lemma}\label{Lem:18IX17-L1}
Assume that the asymptotic flatness condition \eqref{Eq:09IX17-A1} holds. There exists some large $R_0 > 0$ such that, for any $R > R_0$, the distance function $d_R$ is smooth on $M \setminus B_R(0)$.
\end{lemma}

\begin{proof}
By the asymptotic flatness condition \eqref{Eq:09IX17-A1}, we have for large $R$ that the shape operator $S_R$ of $\partial B_R(0)$ satisfies
\[
S_R(X) = \frac{1}{R} X + O(R^{-1-{\red{\delta}} }|X|_g)
	\;.
\]
We hence assume in the proof that $R$ is sufficiently large so that $g(S_R(X),X) >0$ on $\partial B_R(0)$.

Arguing by contradiction, assume that there are a point $p \in \partial B_R(0)$, a normalized geodesic $\gamma \subset M \setminus B_R(0)$ emanating from $p$ and perpendicular to $\partial B_R(0)$, and some point $q = \gamma(\ell) \in M \setminus \bar B_R(0)$, $\ell > 0$, which is the first focal point of $\partial B_R(0)$ along $\gamma$. Then there exists a non-trivial Jacobi field $V$ along $\gamma$ such that $V(0) \in T_p (\partial B_R(0))$, $V'(0) =  S_p(V(0))$ and $V(\ell) = 0$, where $S_p$ is the shape operator of $\partial B_R(0)$ at $p$. Note that $V(0) \neq 0$ by non-triviality.

By the asymptotic flatness condition \eqref{Eq:09IX17-A1}, we have
\begin{align*}
I[V] &\geq \int_0^\ell [|V'|_g^2 - \frac{C}{(R + t)^{2 + {\red{\delta}} }} |V|_g^2]\,dt\\
	&\geq \int_0^\ell [(\frac{d}{dt} |V|_g)^2 - \frac{C}{(R + t)^{2 + {\red{\delta}} }} |V|_g^2]\,dt
		\;,
\end{align*}
where
$C$ denotes some positive constant which depends only on the constant in \eqref{Eq:09IX17-A1}. Thus, by Lemma \ref{Lem:WPIneql}, there exists $R_0 > 0$ such that if $R > R_0$, then $I[V] \geq 0$. On the other hand, by the Jacobi equation, we have
\[
I[V] := \int_0^\ell [|V'|_g^2 - R(\gamma', V, \gamma', V)]\,dt = -g(S_p(V(0)),V(0)) < 0
	\;,
\]
which yields a contradiction. The proof is complete.
\qed
\end{proof}

We next consider a Hessian estimate.

\begin{lemma}\label{Lem:18IX17-L2}
Assume that the asymptotic flatness condition \eqref{Eq:09IX17-A1} holds for some ${\red{\delta}}  \in (0,1)$. There exist $C_1 > 0$ and $R_1 > 0$ such that, for all $R > R_1$, the Hessian of $d_R$ satisfies
\[
\Big(\frac{1}{R + d_R} - \frac{C_1}{(R + d_R)^{1 + {\red{\delta}} }}\Big) h \leq \nabla^2 d_R(x) \leq \Big(\frac{1}{R + d_R} + \frac{C_1}{(R + d_R)^{1 + {\red{\delta}} }}\Big) h
	\;,
\]
where $h$ is the metric induced by $g$ on the level sets of $d_R$.
\end{lemma}

\begin{proof} The proof is standard. Let $R_0$ be as in Lemma \ref{Lem:18IX17-L1}. Assume that $R > R_0$ in the sequel and let $\gamma$ be a normalized geodesic emanating from $\partial B_R$. For $r > 0$, let $\lambda_{\max}(r)$ and $\lambda_{\min}(r)$ be the largest and smallest eigenvalues of $\nabla^2 d_R(\gamma(r))$. Then $\lambda_{\max}$ and $\lambda_{\min}$ are Lipschitz and, in view of  \eqref{Eq:09IX17-A1},  satisfy (see e.g. \cite[p.~175]{Petersen})
\begin{align*}
\lambda_{\max}'(r) + \lambda_{\max}^2(r) &\leq \frac{C_2}{(R + r)^{2+ {\red{\delta}} }}
	, \qquad \lambda_{\max}(0) \leq \frac{1}{R} + \frac{C_2}{R^{1+ {\red{\delta}} }}
	\;,\\
\lambda_{\min}'(r) + \lambda_{\min}^2(r) &\geq - \frac{C_2}{(R + r)^{2+ {\red{\delta}} }}
	, \qquad \lambda_{\min}(0)  \geq \frac{1}{R} - \frac{C_2}{R^{1+ {\red{\delta}} }}
	\;,
\end{align*}
where $C_2$ is some positive constant which depends only on the constant in \eqref{Eq:09IX17-A1}.

Now, fix some $C_1 > \frac{C_2}{1 - {\red{\delta}} }$ and consider the functions
\[
 f_\pm(r) = \frac{1}{r} \pm \frac{C_1}{(R + r)^{1 + {\red{\delta}} }}
	\;.
\]
We have
\[
 f_\pm'(r) +  f_\pm^2(r) = \pm \frac{C_1(1 - {\red{\delta}} )}{(R + r)^{2 + {\red{\delta}} }} + \frac{C_1^2}{(R + r)^{2 + 2{\red{\delta}} }}
 	\;.
\]
Hence, as ${\red{\delta}}  \in (0,1)$, we have for large $R$ that
\begin{align*}
f_+'(r) + f_+^2(r) &\geq \frac{C_2}{(R + r)^{2+ {\red{\delta}} }}
	, \qquad f_+(0) \geq \frac{1}{R} + \frac{C_2}{R^{1+ {\red{\delta}} }}
	\;,\\
f_-'(r) + f_-^2(r) &\leq - \frac{C_2}{(R + r)^{2+ {\red{\delta}} }}
	, \qquad f_-(0)  \leq \frac{1}{R} - \frac{C_2}{R^{1+ {\red{\delta}} }}
	\;.
\end{align*}
A simple first order ODE comparison then yield $\lambda_{\max} \leq f_+$ and $\lambda_{\min} \geq f_-$, which imply the assertion.
\qed
\end{proof}

\subsection{Unique continuation}

Recall that we aim to show that any solution of \eqref{Eq:20VIII17-1} satisfying
\begin{equation}
\omega \in L^2(M)
	\;.
	\label{Eq:11IX17-OFO}
\end{equation}
must be identically zero.

Without loss of generality, we may assume that the boundary of $M$ is some large coordinate sphere $S_0$ of radius $R_0$ near infinity. Let $r$ denote the $g$-distance function to $S_0$, which, by Lemma \ref{Lem:18IX17-L1}, is a smooth function on the exterior of $S_0$. For $t > 0$, let $S_t$ and $\Omega_{t,\infty}$ denote respectively the set $\{r = t\}$ and $\{r > t\}$.

In the sequel, unless otherwise stated, $C$ will denote some positive constant which varies from line to line, but depends only on $R_0$ and the constant in the asymptotic flatness condition \eqref{Eq:09IX17-A1}.

Note that, by applying elliptic estimates to the PDE \eqref{Eq:08IX17-2} on any Euclidean unit ball in the asymptotic region, it is readily seen that \eqref{Eq:11IX17-OFO} implies
\begin{equation}
|\nabla \omega(x)|_g \leq C\sup_{B(x,1/2)} |\omega|_g \leq C\|\omega\|_{L^2(B(x,1))}  \text{ for } x \in \Omega_{1,\infty}
	\;,
	\label{Eq:11IX17-OFO-DX}
\end{equation}
where the constant $C$ depends only on $a$, $R_0$ and the constant in \eqref{Eq:09IX17-A1}. This implies in particular that
\begin{equation}
\nabla \omega \in L^2(M)
	\;.
	\label{Eq:11IX17-OFO-D}
\end{equation}
%

Define
\begin{align}
X(r)
	&= \frac{1}{(r + R_0)^2} \int_{S_r} |\omega|_g^2\,d\sigma_g
	\;,\label{Eq:XDef}\\
E(r)
	&= \frac{1}{(r + R_0)^2} \int_{\Omega_{r,\infty}} \Big[ |\nabla \omega|_g^2 -   a^2  |\omega|_g^2 + \Ric(\omega^\sharp, \omega^\sharp) \Big]\,dv_g
	\;.\label{Eq:EDef}
\end{align}
Note that $E$ is well defined thanks to \eqref{Eq:11IX17-OFO} and \eqref{Eq:11IX17-OFO-D}. Note also that, by \eqref{Eq:08IX17-2},
\[
E(r) = \frac{1}{2(r + R_0)^2}\int_{\Omega_{r,\infty}} \nabla^i \nabla_i |\omega|_g^2\,dv_g
	= 		- \frac{1}{(r + R_0)^3}\int_{S_r} g(\beta, \omega)\,d\sigma_g
		\;,
\]
where
\[
\beta_j = (r + R_0) \nabla^i r\,\nabla_i \omega_j
	\;.
\]
In addition, there exists $C_2 > 0$ depending only on $R_0$ and the constant in \eqref{Eq:09IX17-A1} such that
\begin{align}
\Big|\frac{d}{dr} X(r) + 2E(r)\Big|
	\leq \frac{  C_2 X(r)}{(r + R_0)^{1 + {\red{\delta}} }}
		\;.\label{Eq:09IX17-A3}
\end{align}

\begin{lemma}\label{Lem:09IX17-L1}
Assume that the asymptotic flatness condition \eqref{Eq:09IX17-A1} and the $L^2$ condition \eqref{Eq:11IX17-OFO} hold. There exist $k_0 > 0$ such that for $k > k_0$ one can find some $r_1 > 0$ so that for $r > r_1$,
\begin{multline}
\frac{d}{dr} \Big[(r + R_0)\,\exp\Big(\frac{k}{(r + R_0)^{\red{\delta}} }\Big)\,E(r)\Big] \leq -\frac{2\exp\Big(\frac{k}{(r + R_0)^{\red{\delta}} }\Big)}{(r + R_0)^{3}}\int_{S_r} |\beta|_g^2\,d\sigma_g\\
	- \frac{a^2\exp\Big(\frac{k}{(r + R_0)^{\red{\delta}} }\Big)}{(r + R_0)^2} \int_{\Omega_{r,\infty}}   |\omega|_g^2 \,dv_g
	\;.
	\label{Eq:09IX17-A2}
\end{multline}
In particular, $E(r) \geq 0$ for  $r > r_1$.
\end{lemma}

\begin{proof} For $m \gg R_0$, let $\zeta_m$ be a cut-off function such that $0 \leq \zeta_m \leq 1$ in $M$, $\zeta_m \equiv 1$ in $\{r < m\}$, $\zeta_m \equiv 0$ in $\{r > 2m\}$, and $|\nabla \zeta_m|_g \leq \frac{C}{m}$ in $M$, where here and below $C$ denotes some constant which is independent of $m$.

Keeping in mind the asymptotic flatness and the Hessian estimate Lemma \ref{Lem:18IX17-L2}, we compute for $0 < s \ll m$,
\begin{align*}
 \int_{\Omega_{s,\infty}} \nabla^i \nabla_i \omega_j\,\beta^j\,\zeta_m\,dv_g
	&= \int_{\Omega_{s,\infty}} (r + R_0)\, \nabla^i \nabla_i \omega_j\, \nabla^k r\,\nabla_k \omega^j\,\zeta_m\,dv_g\\
	&\leq -\int_{\Omega_{s,\infty}}  \nabla_i \omega_j\,\Big[\nabla^i ((r + R_0)\,\nabla^k r)\,\nabla_k \omega^j + (r + R_0) \nabla^k r\,\nabla^i \nabla_k \omega^j\Big]\,\zeta_m\,dv_g\\
		&\qquad - \frac{1}{s + R_0} \int_{S_s} |\beta|_g^2\,d\sigma_g
			+ C \int_{\Omega_{m,\infty}} |\nabla \omega|_g^2 \,dv_g\\
	&\leq -\int_{\Omega_{s,\infty}} \Big[ 1 - \frac{C}{(r + R_0)^{{\red{\delta}} }} \Big]|\nabla \omega|_g^2\,\zeta_m dv_g\\
		&\qquad -\int_{\Omega_{s,\infty}}  (r + R_0)\Big[\frac{1}{2}\nabla^k r\, \nabla_k |\nabla \omega|_g^2   - \nabla^k r\,\nabla^i \omega^j\,R_{kijl}\omega^l\Big]\,\zeta_m\,dv_g\\
		&\qquad - \frac{1}{s + R_0} \int_{S_s} |\beta|_g^2\,d\sigma_g
					+ C \int_{\Omega_{m,\infty}} |\nabla \omega|_g^2 \,dv_g\\
	&\leq \int_{\Omega_{s,\infty}} \Big[\frac{1}{2} |\nabla \omega|_g^2 + \frac{C}{(r + R_0)^{{\red{\delta}} }}(|\nabla \omega|_g +  | \omega|_g) |\nabla \omega|_g \Big]\,\zeta_m dv_g\\
		&\qquad - \frac{1}{s + R_0} \int_{S_s} |\beta|_g^2\,d\sigma_g
			+ \frac{1}{2} (s + R_0)  \int_{S_s} |\nabla\omega|_g^2\,d\sigma_g
						+ C \int_{\Omega_{m,\infty}} |\nabla \omega|_g^2 \,dv_g
	\;.
\end{align*}

On the other hand, by the PDE \eqref{Eq:08IX17-2} and the Hessian estimate Lemma \ref{Lem:18IX17-L2},
\begin{align*}
 \int_{\Omega_{s,\infty}} \nabla^i \nabla_i \omega_j\,\beta^j\,\zeta_m\,dv_g
	&= \int_{\Omega_{s,\infty}} \Big[- \frac{1}{2} a^2 (r + R_0) \nabla^k r \nabla_k |\omega|_g^2 + (r + R_0) R_{ij} \omega^i \, \nabla^k r\,\nabla_k \omega_j \Big]\,\zeta_m dv_g\\
	&\geq \int_{\Omega_{s,\infty}} \Big[ \frac{3}{2} a^2 |\omega|_g^2 - \frac{C}{(r + R_0)^{{\red{\delta}} }} |\omega|_g( |\nabla \omega|_g + |\omega|_g )\Big]   \,dv_g\\
		&\qquad + \frac{1}{2} (s + R_0)  \int_{S_s} a^2  |\omega|_g^2 \,d\sigma_g
					- C\int_{\Omega_{m,\infty}} |\omega|_g^2\,dv_g
			\;.
\end{align*}

In addition, we have
\begin{align*}
 \int_{S_s} \Ric(\omega^\sharp, \omega^\sharp)\,d\sigma_g
 	&\geq \int_{S_s} \frac{C}{(r + R_0)^{2 + {\red{\delta}} }} |\omega|_g^2\,d\sigma_g\\
	&= \int_{\Omega_{s,\infty}} \nabla^i \Big(\frac{C}{(r + R_0)^{2 + {\red{\delta}} }} |\omega|_g^2\,\zeta_m\,\nabla_i r\Big)\,dv_g\\
	&\geq  -\int_{\Omega_{s,\infty}} \frac{C}{(r + R_0)^{\red{\delta}} } |\omega|_g( |\nabla \omega|_g + |\omega|_g )\,dv_g
 \,.
\end{align*}

Combining the last three estimates and sending $m \rightarrow \infty$ (using \eqref{Eq:11IX17-OFO} and \eqref{Eq:11IX17-OFO-D}), we can find some $C_1 > 0$ depending only on $R_0$ and the constant in \eqref{Eq:09IX17-A1} such that
\begin{align*}
& \int_{S_s} \Big[|\nabla\omega|_g^2\, -    a^2  |\omega|_g^2 + \Ric(\omega^\sharp, \omega^\sharp)\Big]d\sigma_g  \\
	&\qquad \geq \frac{2}{(s + R_0)^2} \int_{S_s} |\beta|_g^2\,d\sigma_g
			- (s + R_0) E(s) \\
		&\qquad\qquad + \frac{1}{s + R_0}\int_{\Omega_{s,\infty}} \Big[2a^2  |\omega|_g^2 - \frac{C_1}{(r + R_0)^{\red{\delta}} } ( |\nabla \omega|_g^2  + |\omega|_g^2)\Big]\,dv_g
			\;.
\end{align*}
It follows that, for all sufficiently large $s$,
\begin{align*}
&(r + R_0) \exp\Big(-\frac{k}{(r + R_0)^{{\red{\delta}} }}\Big)\frac{d}{dr}  \Big[(r + R_0)\,\exp\Big(\frac{k}{(r + R_0)^{\red{\delta}} }\Big)\,E(r)\Big]\Big|_{r = s} \\
	&\qquad = - (s + R_0)\Big(1 + \frac{k{\red{\delta}} }{(s + R_0)^{{\red{\delta}} }}\Big) E(s)
		- \int_{S_s}  \Big[|\nabla\omega|_g^2\, -    a^2  |\omega|_g^2 + \Ric(\omega^\sharp, \omega^\sharp)\Big]d\sigma_g  \\
	&\qquad \leq -\frac{2}{(s + R_0)^2} \int_{S_s} |\beta|_g^2\,d\sigma_g
		\\
		&\qquad\qquad  -\frac{1}{s + R_0}\int_{\Omega_{s,\infty}}  \Big[\frac{k{\red{\delta}} }{(s + R_0)^{\red{\delta}} }|\nabla \omega|_g^2 + \frac{2(s + R_0)^{\red{\delta}}  - k{\red{\delta}} }{(s + R_0)^{\red{\delta}} } a^2   |\omega|_g^2\Big]\,dv_g\\
		&\qquad\qquad + \frac{1}{s+ R_0}\int_{\Omega_{s,\infty}}  \frac{1}{(r + R_0)^{\red{\delta}} } \Big[C_1 |\nabla \omega|_g^2  +  (C + C_1) |\omega|_g^2\Big]\,dv_g
		\;.
\end{align*}
The conclusion is readily seen.
\qed
\end{proof}

\bigskip
Let
\begin{equation}
F(r) = \frac{(r + R_0)\exp\big(\frac{2k}{(r + R_0)^\delta}\big)E(r)}{X(r)},
	\label{Eq:FreFcDef}
\end{equation}
wherever $X(r) > 0$. In view of \eqref{Eq:09IX17-A3}, $F$ is readily seen as a perturbation of $\frac{-r X'}{2X}$, and so bears some resemblance to the Almgren frequency function \cite{Almgren-1979} for harmonic functions. The following lemma establishes the monotonicity of $F$.

\begin{lemma}\label{Lem:09IX17-L2}
Assume that the asymptotic flatness condition \eqref{Eq:09IX17-A1} and the $L^2$ condition \eqref{Eq:11IX17-OFO} hold. There exist $k > 0$ and $r_1 > 0$ such that
\begin{equation}
\frac{d}{dr} F(r)
\leq - \frac{a^2\exp\Big(\frac{2k}{(r + R_0)^{\red{\delta}} }\Big)}{(r + R_0)^2 X(r)}\int_{\Omega_{r,\infty}} |\omega|_g^2\,dv_g
	\label{Eq:14IX17-E1}
\end{equation}
for all $r > r_1$ satisfying $X(r) > 0$. In particular, we have the dichotomy
\begin{itemize}
\item either $\omega \equiv 0$ in $M$,
\item or $X(r) > 0$ for all $r > r_1$ and the Almgren-type frequency function $F(r)$ is non-increasing for $r \geq r_1$.
\end{itemize}
\end{lemma}

\begin{proof} Let $k_0$ be as in Lemma \ref{Lem:09IX17-L1}. Recall the estimate \eqref{Eq:09IX17-A3}:
\[
\Big|\frac{d}{dr} X(r) + 2E(r)\Big|
	\leq \frac{  C_2 X(r)}{(r + R_0)^{1 + {\red{\delta}} }}
		\;.
\]
%
Fix some $k > \max(k_0, C_2)$. By \eqref{Eq:09IX17-A2} and \eqref{Eq:09IX17-A3}, we have for large $r$ that
\begin{align*}
&\frac{X^2(r)\exp\Big(-\frac{2k}{(r + R_0)^{\red{\delta}} }\Big)}{r + R_0} \frac{d}{dr} F(r)
\\
	&\qquad= \frac{X^2(r)\exp\Big(-\frac{2k}{(r + R_0)^{\red{\delta}} }\Big)}{r + R_0} \frac{d}{dr} \frac{(r + R_0) \exp\Big(\frac{k}{(r + R_0)^{\red{\delta}} }\Big) E(r)}{\exp\Big(-\frac{k}{(r + R_0)^{\red{\delta}} }\Big)X(r)} \\
	&\qquad\leq -\frac{1}{(r + R_0)^{6}}  \Big(2\int_{S_r} |\beta|_g^2\,d\sigma_g + a^2 (r + R_0)\int_{\Omega_{r,\infty}} |\omega|_g^2\,dv_g\Big)\, \int_{S_r} |\omega|_g^2\,d\sigma_g\\
		&\qquad\qquad - E(r)\,\Big(\frac{d}{dr} X(r) + \frac{k{\red{\delta}} }{(r + R_0)^{1 + {\red{\delta}} }} X\Big)\\
	&\qquad\leq -\frac{1}{(r + R_0)^{6}}  \Big(2\int_{S_r} |\beta|_g^2\,d\sigma_g + a^2 (r + R_0)\int_{\Omega_{r,\infty}} |\omega|_g^2\,dv_g\Big)\, \int_{S_r} |\omega|_g^2\,d\sigma_g\\
		&\qquad\qquad - E(r)\,\Big(-2E(r) + \frac{ k{\red{\delta}}  - C_2}{(r + R_0)^{1+{\red{\delta}} }} X(r)\Big)\\
	&\qquad= -\frac{2}{(r + R_0)^{6}} \Big\{ \int_{S_r} |\beta|_g^2\,d\sigma_g\, \int_{S_r} |\omega|_g^2\,d\sigma_g - \Big(\int_{S_r} g(\beta,\omega)\,dv_g\Big)^2\Big\}\\
		&\qquad\qquad - \frac{a^2 X(r)}{(r + R_0)^3} \int_{\Omega_{r,\infty}} |\omega|_g^2\,dv_g
			 - \frac{k{\red{\delta}}  - C_2}{(r + R_0)^{1+{\red{\delta}} }} E(r)\,X(r)
		\;.
\end{align*}
As $k > C_2$ and $E(r) \geq 0$ for large $r$ (thanks to Lemma \ref{Lem:09IX17-L1}), estimate \eqref{Eq:14IX17-E1} follows from Cauchy-Schwarz' inequality.

We turn to the proof of the stated dichotomy. Assume by contradiction that $\omega \not\equiv 0$ in $M$ but $X(r_2) = 0$ for some $r_2 > r_1$. By unique continuation, $X$ cannot be identically zero in a non-empty open subinterval of $(r_1,r_2)$. Thus, the set $\{r \in (r_1,r_2): X(r) \neq 0\}$ is a union of pairwise disjoint open subintervals of $(r_1, r_2)$. Let $(r_-, r_+)$ be a connected component of this set, so that $X(r_+) = 0$. ($X(r_-)$ is also zero unless $r_- = r_1$, but we will not need this fact.) In this interval, the function $F$ is well-defined and is non-increasing thanks to \eqref{Eq:14IX17-E1}.

Define
\[
F_+ = \lim_{r \rightarrow r_+} F(r) \geq 0
	\;.
\]
Then, for $r \in (r_-,r_+)$, we have $F(r) \geq F_+$ and so
\[
E(r) \geq \frac{F_+}{r + R_0}\,\exp\Big(-\frac{2k}{(r + R_0)^{\red{\delta}} }\Big)\,X(r) \geq  \frac{F_+}{r + R_0}\,\Big(1 - \frac{2k}{(r + R_0)^{\red{\delta}} }\Big)\,X(r)
	\;.
\]
Recalling \eqref{Eq:09IX17-A3}, we have
\[
\frac{d}{dr} X(r) \leq - 2E(r) + \frac{C_2 X(r)}{(r + R_0)^{1 + {\red{\delta}} }} \leq \Big[-\frac{2F_+}{r + R_0} + \frac{C_2 + 4kF_+}{(r + R_0)^{1 + {\red{\delta}} }}\Big]\,X(r)
	\;,
\]
which is equivalent to
\[
\frac{d}{dr} \Big[(r + R_0)^{2F_+} \exp \Big(\frac{C_2 + 4kF_+}{{\red{\delta}} (r + R_0)^{{\red{\delta}} }}\Big)X(r)\Big] \leq 0
	\text{ for }r \in (r_-,r_+)
		\;,
\]
But this is impossible as $X(r_+) = 0$ and $X(r) > 0$ for $r \in (r_-,r_+)$. This contradiction proves the desired dichotomy.
\qed
\end{proof}

Assume that $\omega \not\equiv 0$. Then $X(r) > 0$ for large $r$. Let $F$ be the Almgren frequency function defined by \eqref{Eq:FreFcDef}
 and set
\begin{equation}
F(\infty) := \lim_{r \rightarrow \infty} F(r)
	\;.\label{Eq:11IX17-R2}
\end{equation}

The next two statements show that, roughly speaking, $X$ decays like $r^{-2F(\infty)}$.

\begin{corollary} \label{Cor:11IX17-C2}
Assume that the asymptotic flatness condition \eqref{Eq:09IX17-A1} and the $L^2$ condition \eqref{Eq:11IX17-OFO} hold and that $\omega \not\equiv 0$ in $M$. Then, there exists some $C > 0$ such that
\[
X(r) \leq C (r + R_0)^{-2F(\infty)} \text{ for } r > 0
	\;.
\]
%
%
\end{corollary}

\begin{proof} By Lemma \ref{Lem:09IX17-L2}, there exists some $r_1 > 0$ such that $F(r) \geq F(\infty)$ for all $r > r_1$. We can then follow the argument in the second half of the proof of Lemma \ref{Lem:09IX17-L2} (with $F_+$ replaced by $F(\infty)$) to show that
%
\[
\frac{d}{dr} \Big[(r + R_0)^{2F(\infty)} \exp \Big(\frac{C_2 + 4kF(\infty)}{{\red{\delta}} (r + R_0)^{\red{\delta}} }\Big)X(r)\Big] \leq 0
	\text{ for } r > r_1
		\;.
\]
The conclusion follows.
\qed
\end{proof}

\begin{corollary}
  \label{Cor:09IX17-CX}
Assume that the asymptotic flatness condition \eqref{Eq:09IX17-A1} and the $L^2$ condition \eqref{Eq:11IX17-OFO} hold and that $\omega \not\equiv 0$ in $M$. Then, for any $\gamma > 2F(\infty)$,
\[
\liminf_{r \rightarrow \infty} r^\gamma\,X(r) = \infty
	\;.
\]
\end{corollary}

\begin{proof} Assume by contradiction that $\liminf_{r \rightarrow \infty} r^\gamma\,X(r) < \infty$. Replacing $\gamma$ by a smaller number, which is still larger than $2F(\infty)$,
if necessary, we may assume without loss of generality that $\liminf_{r \rightarrow \infty} r^\gamma\,X(r) = 0$.

Let $k$ and $r_1$ be as in Lemma \ref{Lem:09IX17-L2} and fix some $s > r_1$ such that $2F(s) < \gamma$. Then $F(r) \leq F(s) < \frac{1}{2}\gamma$ for all $r \geq s$. In view of  \eqref{Eq:09IX17-A3}, we have for $r \geq s$ that
\[
\frac{X'(r)}{X(r)} \geq -\frac{2\exp\Big(-\frac{2k}{(r + R_0)^{\red{\delta}} }\Big)}{r + R_0} F(r) - \frac{C_2}{(r + R_0)^{1 + {\red{\delta}} }} \geq - \frac{2}{r + R_0} F(s) -  \frac{C_2}{(r + R_0)^{1 + {\red{\delta}} }}
	\;,
\]
and so
\[
\frac{d}{dt} \Big[ X(r)\,(r + R_0)^{2F(s)} \exp\Big(- \frac{C_2}{{\red{\delta}} (r + R_0)^{\red{\delta}} }\Big)\Big] \geq 0
	\;.
\]
As $2F(s) < \gamma$, we have by assumption that the function in the square bracket tends to zero along a sequence $r_i \rightarrow \infty$. It follows that $X(r) = 0$ for all $r > s$, which contradicts the fact that $X(r) > 0$ for large $r$. The proof is complete.
\qed
\end{proof}

We next show that, under the $L^2$ assumption on $\omega$, $X$ must decay faster than $r^{-3}$.

\begin{corollary}\label{Cor:14IX17-C1}
Assume that the asymptotic flatness condition \eqref{Eq:09IX17-A1} and the $L^2$ condition \eqref{Eq:11IX17-OFO} hold and that $\omega \not\equiv 0$ in $M$. Then $F(\infty) \geq \frac{3}{2}$.
\end{corollary}

\begin{proof}
By Corollary \ref{Cor:09IX17-CX}, for any $\varepsilon > 0$, there exists $r_\varepsilon$ such that
\[
X(r) \geq (r + R_0)^{-2F(\infty) - \varepsilon} \text{ for } r > r_\varepsilon
	\;.
\]
Thus, by \eqref{Eq:11IX17-OFO}, we have
\[
\infty > \int_{r_\varepsilon}^\infty (r + R_0)^2\,X(r)\,dr \geq \int_{r_\varepsilon}^\infty (r + R_0)^{2 - 2F(\infty) - \varepsilon}\,dr
	\;.
\]
This implies that $2F(\infty) + \varepsilon > 3$. Sending $\varepsilon \rightarrow 0$ we obtain the conclusion.
\qed
\end{proof}

We now wrap up the argument, by showing, using the estimate for $F'$ of Lemma \ref{Lem:09IX17-L2}, that $X$ cannot decay faster than $r^{-3}$, unless $\omega \equiv 0$.

\begin{proposition}\label{Prop:19I18-P1}
Let $(M^3,g)$ be an
asymptotically flat three dimensional Riemannian manifold with or without boundary satisfying the asymptotic flatness condition \eqref{Eq:09IX17-A1}. If a one-form $\omega \in L^2(\Mext)$ satisfies \eqref{Eq:08IX17-2}
on $M$ for some non-zero $a \in \RR \setminus\{0\}$,
then $\omega \equiv 0$.
\end{proposition}

\begin{proof} Assume by contradiction that $\omega \not \equiv 0$. By Lemma \ref{Lem:09IX17-L2}, there exists some $r_1$ such that $X(r) > 0$ for all $r > r_1$.

Let $F(\infty)$ be given by \eqref{Eq:11IX17-R2}. By Corollary \ref{Cor:14IX17-C1}, $F(\infty) \geq \frac{3}{2}$.

Fix some $\varepsilon > 0$. By Corollaries \ref{Cor:11IX17-C2} and \ref{Cor:09IX17-CX}, there exists some $r_\varepsilon > 0$ such that
\[
(r + R_0)^{-2F(\infty) -\varepsilon} \leq X(r) \leq C(r + R_0)^{-2F(\infty)} \text{ for } r > r_\varepsilon
	\;,
\]
which implies that
\[
\frac{1}{X(r)} \int_{\Omega_{r,\infty}} |\omega|_g^2\,dv_g
	= \frac{1}{X(r)} \int_{r}^{\infty} (r + R_0)^2\,X(r)\,dr \geq \frac{1}{C} (r + R_0)^{3 - \varepsilon}
	\text{ for } r > r_\varepsilon
		\;.
\]
Returning to the estimate \eqref{Eq:14IX17-E1} in Lemma \ref{Lem:09IX17-L2}, we obtain
\[
\frac{d}{dr} F(r) \leq - \frac{a^2}{C} (r + R_0)^{1 - \varepsilon} \text{ for } r > \max(r_1, r_\varepsilon)
	\;.
\]
As $F$ is bounded, this is impossible for small $\varepsilon$. This contradiction finishes the proof.
\qed
\end{proof}

\section{Stationary case}
 \label{s19IX17.4}

In this section we will prove non-existence of stationary non-inheriting solutions. This will follow from a general result concerning $L^2$-solutions of PDEs arising from
elliptic second order Laplacian-type operators, Corollary~\ref{C21XII17.1} below. As already mentioned, the results here also cover the static case of Section~\ref{s19IX17.2} where, however, the result was obtained by a more elementary argument. We believe that both proofs have interest in their own.

Similarly to \eqref{Eq:09IX17-A1}, we will assume that there exists a constant $\delta>0$ such that
%
\begin{equation}\label{21XII17.31}
 g_{ij}-\delta_{ij} =
  O_k(r^{-\delta})
  \,,
  \quad
  V-1 =
  O_k(r^{-\delta})
  \,, \quad
  \theta_i =
  O_k(r^{-\delta})
  \,,
\end{equation}
for some $k$ large enough, which we leave unspecified. Here we write $f=
O_k(r^{-\delta})$ if for all multi-indices $\alpha$ with $0\le |\alpha|\le k$ we have
$$
  r^{|\alpha|} \partial^\alpha f = O(r^{-\delta})
$$
for large $r$.

Assuming \eqref{21XII17.31} and $V>0$
one readily checks that the stationary inheriting equations \eqref{26IX17.13}
satisfy the hypotheses of Corollary~\ref{C21XII17.1} below, leading to

\begin{thm}
 \label{T21XII17.1}
There are no asymptotically flat solutions of the Maxwell-Einstein equations with a stationary metric and non-inheriting Maxwell fields.
\end{thm}

\begin{remark}
{\rm
In fact, our proof 
applies to the more general class of metrics \eqref{23XII17.1}-\eqref{9XII17.1} below.
}
\end{remark}

We emphasise that no global hypotheses on the solution are imposed other than the existence of an asymptotically flat end.

The operators covered by our analysis below include those of the form $\Delta_g-\lambda$,
$\lambda>0$, where
\begin{equation}\label{23II18.1}
     \Delta_g= - g^{ij} \nabla _i \nabla_j
\end{equation}
is the Laplacian of an asymptotically Euclidean metric, and
their modifications by   decaying   lower order terms. This
includes operators acting on sections of vector bundles which behave
like the Laplacian both in terms of the highest derivatives and in
terms of the asymptotic behavior of the coefficients at the boundary,
i.e.\ at metric infinity. We show that such solutions in fact
vanish identically.

The methods are Carleman-type estimates, phrased in a way that was
used in \cite{Vasy:Exponential} to analyze geometric $N$-body problems,
showing unique continuation at infinity for all second order
perturbations of the Laplacian.  These in
turn were motivated by
the closely related works of Froese and Herbst \cite{FroExp}
and the unique continuation theorems at infinity discussed in
\cite{Hormander:Uniqueness} and \cite[Theorem~17.2.8]{Hor}. The
key estimates arise from a positive commutator estimate for the exponentially
conjugated Hamiltonian, which is closely related to H\"ormander's
solvability condition for PDE's \cite{Hormander:Differential,
Hormander:Solutions, Duistermaat-Sjostrand:Global}; see
\cite{Zworski:Numerical} for a
discussion, including the relationship to numerical computation

The operators $H$ that we consider below will be
acting on sections of a vector bundle $E$ which is equipped with a
Hermitian metric.
We assume that $H=\Delta\otimes\Id_E+V$ where
$\Delta \equiv \Delta_g$
is
as in \eqref{23II18.1}, and recall that $g$ satisfies $g-g_0\in
S^{-\delta}(\Xb;\Tsc^*\Xb^{\otimes^2_s})$ for some $\delta>0$,
and $V\in \Diffscc^{1,\delta}(\Xb;E)=S^{-\delta}\Diffsc^1(\Xb;E)$ is formally self-adjoint; see Section~\ref{s25X17.1} for notation.

We prove the following results:
The first theorem states that positive energy (eigenvalue) eigenfunctions
of $H$ decay superexponentially. There is a simple modification of
the proof to show exponential decay for negative energy eigenfunctions, at any
rate $\alpha$ where $\alpha^2<-\lambda$, $\lambda<0$ the
eigenvalue; this corresponds to the $N$-body result in which case the
decay rate is given by the square root of the distance to the next
threshold above $\lambda$, which is $0$ if $\lambda<0$, and
non-existent (can be considered as $+\infty$) if $\lambda>0$. This fact also explains why it is natural to consider
the unique continuation theorem separately, namely why
super-exponential decay assumptions are natural there; unique
continuation, given superexponential decay, holds for {\em arbitrary}
eigenvalues.

In fact, we show a version of the aforementioned results by only assuming that the equation holds on a collar neighborhood of the boundary at infinity. So let   $\Xbext$ be a collar neighborhood $\{0\le x \le x_1\}$, for some $x_1>0$, of the boundary at infinity $\partial \Xb = \{x=0\}$. In the case of asymptotically flat metrics we have $x=1/r$, and the interior of $\Xbext$ coincides with the asymptotically flat region $\{r> R_1:= 1/x_1\}$.
We have (compare \cite[Theorem~3.1]{Vasy:Exponential}, \cite[Proposition~B.2]{Vasy:Propagation-2} and
\cite[Theorem 2.1]{FroExp}):

\begin{thm}
Let $\lambda>0$, and suppose that
$\psi\in L^2_{\scl}(\Xbext)$ satisfies $H\psi=\lambda\psi$. Then $e^{\alpha/x}\psi
\in L^2_{\scl}(\Xbext)$ for all $\alpha\in\Real$.
\end{thm}

The unique continuation theorem at infinity is the following; note
that $\lambda$ is now arbitrary (see \cite[Theorem~4.1]{Vasy:Exponential}, \cite[Proposition~B.3]{Vasy:Propagation-2} and
\cite[Theorem 3.1]{FroExp} for related results):

\begin{thm}
Let $\lambda\in\Real$.
If
$H\psi
=\lambda\psi$, $\exp(\alpha/x)\psi\in L^2_{\scl}(\Xbext)$ for all $\alpha$,
then $\psi=0$.
\end{thm}

As an immediate corollary we deduce the absence of positive eigenvalues
for first order perturbations
of $\Delta_{g}$:

\begin{cor}
 \label{C21XII17.1}
Let $\lambda>0$. Suppose that $H\psi
=\lambda\psi$, $\psi\in L^2_{\scl}(\Xbext)$.
Then $\psi\equiv 0$.
\end{cor}

\subsection{Definitions}
 \label{s25X17.1}

Recall from \cite{RBMSpec} the notion of a (short-range)  {\em scattering metric}.
Namely, with $\Xb$ is a manifold with boundary and $x$ is
a boundary defining function on $\Xb$, a (short-range) scattering metric $g_0$
is a Riemannian metric on $\red{M}:=\Xb^\circ$ which is of the form
\bel{23XII17.1}
 g_0=x^{-4}\,dx^2+x^{-2}h
\ee
near $\pa \Xb$, where $h$ is a symmetric
two-covariant tensor that restricts to a metric on $\pa \Xb$.
In this work we allow manifolds $(\red{M},g)$ which
might have several boundary components and asymptotic ends, with unspecified behaviour there except for one end where the metric is asymptotically flat, cf.\ \eqref{23XII17.1}. The discussion that follows applies only to such asymptotic regions.

As such, the metric $g_0$ as in \eqref{23XII17.1}
is a
smooth (on $\Xb$)
section of the second symmetric power of the scattering cotangent
bundle $\Tsc^*\Xb$, with some additional product structure. Indeed, near $\pa\Xb$, a general
smooth section of $\Tsc^*\Xb$ is a linear combination, with
$\CI(\Xb)$-coefficients, of $\frac{dx}{x^2},\frac{dy_j}{x}$, where
$y_j$ are local coordinates on $\pa \red{M}$, thus a general smooth section of
this bundle if a linear combination of
$\frac{dx^2}{x^4},\frac{dx\,dy_j}{x^3},\frac{dy_idy_j}{x^2}$; these
short range metrics thus have $1+O(x^2)$ for the coefficient of
$\frac{dx^2}{x^4}$, and $O(x)$ for the coefficient of $\frac{dx\,dy_j}{x^3}$.

Let $\Delta_{g_0}$
be the Laplacian of this metric. This is a typical element of
$\Diffsc(\Xb)$, the algebra of scattering differential operators. Here ``sc'' stands for ``scattering''.
The latter is generated, over $\Cinf(\Xb)$, by the vector fields
$\Vsc(\Xb)=x\Vb(\Xb)$; $\Vb(\Xb)$ being the Lie algebra of $\Cinf$ vector
fields on $\Xb$ that are tangent to $\pa \Xb$. Thus, in local
coordinates, and over $\CI(\Xb)$ (locally), $\Vsc(\Xb)$ is spanned by
$x^2\pa_x$ and $x\pa_{y_j}$, while $\Vb(\Xb)$ is spanned by $x\pa_x$
and $\pa_{y_j}$. We work with what might
be called a (very) long-range scattering metric, namely we assume that
$g$ is a conormal to $\pa \red{M}$, or in other words symbolic,
real
section of $\Tsc^*\Xb^{\otimes^2_s}$, of order $0$ with the extra
property that
\begin{equation}\label{9XII17.1}
 g-g_0\in S^{-\delta}(\Xb;\Tsc^*\Xb^{\otimes^2_s})
  \,,
  \ \text{for
some $ {0<\delta\le 1}$.}
\end{equation}
 Recall that, in this compactified notation, a symbol,
$a\in S^\alpha$,
of ``symbolic order'' $\alpha$, i.e. growth rate, is one satisfying that for all $P\in\Diffb(\Xb)$
(equivalently, for all finite, possibly empty, products $P$ of elements of $\Vb(\red{M})$),
$x^{\alpha}Pa\in L^\infty(\red{M})$;
 thus our very long range metrics in
particular allow $O(x^\delta)$ coefficients for
$\frac{dx\,dy_j}{x^3}$, and $1+O(x^\delta)$ for the coefficients for
$\frac{dx^2}{x^4}$. (The `long-range' metrics would have smooth
$1+O(x)$ $\frac{dx^2}{x^4}$ terms, i.e.\ would only have a product
structure {\em exactly at} $\pa \red{M}$, and would include asymptotically
Schwarzschild metrics; \red{our} `very long-range scattering
  metrics' satisfy even weaker conditions.)  We write
$$
 S^m\Diffsc^k(\Xb)=\Diffscc^{k,-m}(\Xb)=S^m\otimes\Diffsc^k(\Xb)
$$
for
scattering differential operators with such conormal/symbolic coefficients.

Classical symbols in $S^{\alpha}$ are symbols that have a one-step polyhomogeneous expansion, i.e.,
$\sum_{j\in \NN} x^{j-\alpha} f_j(y)
$, considered as an asymptotic summation. See \cite{RBMSpec} for more details; we follow closely the notations there.
In particular we write $L^2_{\scl}(\Xb)$ for $L^2(\red{M})$ with the measure induced from the asymptotically flat metric.

The operator space $S^m\Diffsc^k(\Xb)$ form a filtered algebra,
so
$$
A\in S^m\Diffsc^k(\Xb),\ B\in S^{m'}\Diffsc^{k'}(\Xb)
\
 \Longrightarrow
 \
 AB \in S^{m+m'}\Diffsc^{k+k'}(\Xb),
$$
which is in fact commutative to leading order both in terms of the
differentiability and the growth orders, so
\begin{equation}\label{13XII17.11}
[A,B]\in S^{m+m'-1}\Diffsc^{k+k'-1}(\Xb).
\end{equation}

It may help the reader if we explain why the Euclidean setting is a particular
example of this setup.
There $\red{M}$ is a vector space with a metric $g_0$,
which can hence by identified with $\Rn$. Moreover, $\Xb$ is the radial
(or geodesic) compactification of $\Rn$ to a ball. Explicitly, this
arises by considering `inverse' polar coordinates, and writing $w\in \red{M}$ as
$w=r\omega=x^{-1}\omega$, $\omega\in\Sn$, so $x=|w|^{-1}$,
e.g.\ in $|w|\geq 1$, and attaching $x=0$, i.e.\ $\{0\}_x\times\Sn$ to
$(0,1]_x\times S^n$ by simply extending the range of $x$. In particular,
$\pa \Xb$ is given by $x=0$, i.e.\ it is just $\Sn$. The metric $g_0$
then has the form $dr^2+r^{-2}h_0=x^{-4}dx^2+x^2h_0$, where $h_0$ is the
standard metric on $\Sn$, so $(\Xb,g_0)$
fits exactly into this framework. Then
$\Tsc \Xb,\Tsc^*\Xb$ are trivial vector bundles over
$\Xb$; namely $\Tsc^*\Xb=\Xb\times \red{M}^*$, $\red{M}^*$ being
the dual vector space of $\red{M}$. Thus, in terms of the coordinates $w_j$,
$\pa_{w_j}$ span $\Vsc(\Xb)$ over $\CI(\Xb)$, i.e.\ with classical
symbolic coefficients of order $0$ on $\red{M}$. This in particular shows
that the scattering differential operator algebra is the geometric
generalization of the algebras considered by Parenti
\cite{Parenti:Operatori} and
Shubin~\cite{Shubin:Pseudodifferential}. On the other hand, in the
complement of the origin, say in the region where $w_n>\ep|w_j|$,
$j\neq n$, $\ep>0$, $w_n\pa_{w_i}$, $i=1,\ldots,n$, span $\Vb(\Xb)$ in
the similar sense, which shows why the above description of symbolic
regularity is {\em exactly} the standard one on the vector space $\red{M}$.

\subsection{Sketch of proofs}

It might be helpful to the reader to provide an extended outline of proofs; the details will follow.
The rough idea
is to conjugate by exponential weights
$e^F$, where $F$ is a symbol of growth order $1$, for example $F=\alpha/x$ for small $x$.
If $\psi$ is an eigenfunction of $H$ of eigenvalue $\lambda$, then
$\psi_F=e^F\psi$ solves
\begin{equation*}
P\psi_F=0\Mwhere P=H(F)-\lambda=e^FHe^{-F}-\lambda.
\end{equation*}
Now let $\re P:=(P+P^*)/2$, where $*$ denotes formal adjoint relative to $L^2$
 is given by $H-\alpha^2-\lambda$, while $\im P:=(P-\re P)/\sqi $ is given by
$-2\alpha(x^2 D_x)$, modulo $x\Diffsc(\Xb)$; the notation is justified
by the principal symbol of $\re P$ being the real part of the
principal symbol of $P$, etc. Here, and throughout this section,
$$
 D_k = \frac 1{\sqi } \partial_k
 \,,
$$
with $\sqi =\sqrt{-1}$.
By elliptic regularity,
using $P\psi_F=0$, $\|\psi_F\|_{x^p\Hsc^k(\Xb)}$ is bounded by
$C_{k,p}\|\psi_F\|_{x^pL^2_{\scl}(\Xb)}$, so the order of various
differential operators  is irrelevant for the purpose of norm estimates, while the weight is important.
Since
\begin{equation*}
P^*P=(\re P)^2+(\im P)^2+i(\re P\im P-\im P\re P),
\end{equation*}
so
\begin{equation}\label{eq:comm-8-0}
0=(\psi_F,P^*P\psi_F)=\|\re P\psi_F\|^2+\|\im P\psi_F\|^2+
(\psi_F,\sqi [\re P,\im P]\psi_F).
\end{equation}
Now, being a commutator,
 $[\re P,\im P]\in x\Diffsc(\Xb)$, i.e.\ has an extra
order of vanishing, which shows that
\begin{equation*}
\|\re P\psi_F\|\leq C_1\|x^{1/2}\psi_F\|,
\ \|\im P\psi_F\|\leq C_1\|x^{1/2}\psi_F\|.
\end{equation*}
Here, and elsewhere, $\|\cdot \|$ denotes the $L^2$-norm with respect to the standard measure associated with the metric $g$, and $(\cdot,\cdot)$ the associated scalar product.
Due to the extra factor of $x^{1/2}$, this can be interpreted roughly
as saying that
$\psi_F$ is, in an asymptotic (decay) sense, `almost' in the nullspace of both $\re P$ and of
$\im P$, hence both of $H-\lambda-\alpha^2$ and $x^2D_x$.

If, moreover, $(\psi_F,\sqi [\re P,\im P]\psi_F)$ is positive, modulo
terms involving $\re P$ and $\im P$ (which can be absorbed in the
squares in \eqref{eq:comm-8-0}), and terms of the form $(\psi_F,R\psi_F)$,
$R\in x^{1+\delta}\Diffscc^{*,0}(\Xb)$ (with $*$ showing that the
differential order is irrelevant due to elliptic regularity), which are thus bounded by $C_2\|x^{(1+\delta)/2}\psi_F\|^2$,
then the factor $x^{(1+\delta)/2}$ (which has an extra $x^{\delta/2}$ compared to
$\|x^{1/2}\psi_F\|$) yields easily a bound for $\|x^{1/2}\psi_F\|$ in
terms of $\|\psi\|$. This gives estimates for the
norm $\|x^{1/2}\psi_F\|$, uniform both in $F$ and in $\psi$. A regularization
argument in $F$ then gives the exponential decay of $\psi$.

The positivity of $(\psi_F,\sqi [\re P,\im P]\psi_F)$, in the sense described
above, is easy to see if we replace $\sqi [\re P,\im P]$ by
$\sqi [H-\lambda-\alpha^2,-2\alpha x^2D_x]$: this commutator is a standard one
considered in $N$-body scattering, although the even more usual one
would be $\sqi [H-\lambda-\alpha^2,-2xD_x]$, whose local positivity
in the spectrum of $H$ is the Mourre estimate
\cite{Mourre:Operateurs, Perry-Sigal-Simon:Spectral,
FroMourre}. Indeed, the latter
commutator is the one considered by Froese and Herbst in Euclidean
$N$-body potential scattering, and we could adapt their argument
(though we would need to deal with numerous error terms) to our setting.
However, the argument presented here is more robust, especially in the
high energy sense discussed below, in which their approach would not
work in the generality considered here. There is one exception: for
$\alpha=0$, $\im P$ degenerates, and in this case we need to `rescale'
the commutator argument, and consider $\sqi [H-\lambda-\alpha^2,-2xD_x]$
directly.

We next want to let
$\alpha\to\infty$. Since most of the related literature considers
``semiclassical problems'', we let $h=\alpha^{-1}$, and replace $P$ above
by $P_h=h^2P$, which is a semiclassical differential operator, $P_h
\in\Diff_{\sccl,h}^2(\Xb)$. Here
$\Diff_{\scl,h}(\Xb)$ is the algebra of semiclassical scattering
differential operators discussed, for example, in
\cite{Vasy-Zworski:Semiclassical} in this setting (see
\cite{Zworski:Semiclassical} for a general introduction to
semiclassical analysis), and $\Diff_{\sccl,h}(\Xb)$ is its
conormal/symbolic coefficient version. The space $\Diff_{\scl,h}(\Xb)$ is generated by $h\Vsc(\Xb)$ over
$\Cinf(\Xb\times[0,1)_h)$.
In this semiclassical sense, the first and
zeroth order terms in $H$ do not play a role in $P_h$: their
contribution is in $h\Diff_{\scl,h}^1(\Xb)$, hence their contribution
to the commutator $\sqi [\re P_h,\im P_h]$ is in $xh^2\Diff_{\scl,h}(\Xb)$.
Moreover, at infinity $\sqi [\re P_h,\im P_h]$ is close
to the corresponding commutator with $P_h$ replaced by $h^2(e^F \Delta_{g_0}
e^{-F}-\lambda)$. Since in the latter case the commutator is
positive, modulo terms than can be absorbed in the
two squares in \eqref{eq:comm-8-0}, $\sqi [\re P_h,\im P_h]$ is also positive
for $g$ near $g_0$, which automatically holds near infinity (where
this is relevant). This gives an estimate as above, from which
the vanishing of $\psi$ near $x=0$ follows easily.

We remark that the estimates we use are related
to the usual proof of unique continuation at infinity on $\Rn$
(i.e.\ not
in the $N$-body setting), see \cite[Theorem~17.2.8]{Hor}, and to
H\"ormander's solvability condition for PDE's in terms of the
real and imaginary parts of the principal symbol.
Indeed, although
in \cite[Theorem~17.2.8]{Hor} various changes of coordinates are used first,
which change the nature of the PDE at infinity,
ultimately the necessary estimates also arise from a commutator of the
kind $\sqi [\re P,\im P]$. However, even in that setting,
the proof we present appears more natural
from the point of view of scattering than the one presented there,
which is motivated by unique continuation at points in $\Real^n$.
We remark
that related estimates, obtained by different techniques, form the
backbone of the
unique continuation results of Jerison and Kenig
\cite{Jerison-Kenig:Unique, Jerison:Carleman}.

The true flavor of our arguments is most
clear in the proof of the unique continuation theorem,
Theorem~\ref{thm:unique}. The reason is that on the one hand there is
no need for regularization of $F$, since we are assuming super-exponential
decay, on the other hand the positivity of $\sqi [\re P_h,\im P_h]$ is
easy to see.

The structure of this part of our work is the following. In Section~\ref{sec:prelim} we discuss
various preliminaries, including the structure of the conjugated
Hamiltonian and a Mourre-type global positive commutator estimate.
In Section~\ref{sec:poly} we prove polynomial, and then in
Section~\ref{sec:exp} the exponential decay of eigenfunctions with $\lambda>0$.
In Section~\ref{sec:unique}, we prove the unique continuation theorem at
infinity.
We emphasize that the presence of bundles such as $E$ makes
no difference in the discussion, hence they are ignored in order to
keep the notation manageable.

\subsection{Preliminaries}\label{sec:prelim}

We first remark that, for the metrics $g$ under consideration,
the Riemannian measure density takes the form
\begin{equation}\label{eq:dg-form}
 dg=\sqrt{\det (g_{ij})}\,dx\,dy=\tilde g\,\frac{dx\,dy}{\red{x^{n+1}}}\,,\   n=\dim \red{M},
 \ \tilde g\in\Cinf(\Xb)+S^{-\delta}(\Xb)
  \,.
\end{equation}
By our conditions \eqref{9XII17.1} on the form of $g$, the Laplacian
takes the following form
\begin{equation}
 \label{P9XII17.2}
\Delta_g=(x^2 D_x)^2
  +\sum_j b_j x^2 P_j+x^{\delta}R
\end{equation}
with $
b_j\in\Cinf(\Xb)$,
$P_j\in\Diff^2(\bXb)$, $R\in S^{0}\Diffsc^2(\Xb)$, and with all sums over finite sets of indices.  (Recall that $D_x = \frac 1{\sqi } \partial_x$.)
Hence, $H=\Delta_g+V$ takes the form
\begin{equation*}
H=(x^2 D_x)^2+\sum_j b'_j x^2 P'_j+x^{\delta}R',
\end{equation*}
with $b'_j\in\Cinf(\Xb)$,
$P'_j\in\Diff^2(\bXb)$, $R'\in \Diffsc^2(\Xb)$.
%

Below we consider the conjugated Hamiltonian $H(F)=e^{F}He^{-F}$, where
$F$ is a symbol of growth  order $1$. The exponential weights will facilitate
exponential decay estimates, and eventually the proof of unique
continuation at infinity.
Let $x_0=\sup_{\Xb} x$. By altering $x$ in a compact subset of $\red{M}$, we
may assume that $x_0<1/2$; we do this for the convenience of notation below.
We let
$S^m([0,1)_x)$ be the space of all symbols $F$ of order (growth rate) $m$ on $[0,1)$,
which satisfy $F\in\Cinf((0,1))$, vanish on $(1/2,1)$, and for which
$\sup_{x\in(0,1)}|x^{m+k}\partial_x^k F|<\infty$ for all $k$. The topology of $S^m
([0,1))$ is given by the seminorms $\sup|x^{m+k}\partial_x^k
F|$. Also, as already mentioned,
the spaces $S^m(\Xb)$ of symbols is defined
similarly, i.e.\ it is
given by seminorms $\sup_{\Xb}|x^m P F|$, $P\in\Diffb^k(\Xb)$.

We have:

\begin{lemma}\label{lemma:FH-1}
    Suppose $\lambda\in\Real$,
$H\psi=\lambda\psi$, $\psi\in L^2_{\scl}(\Xb)$.
Then with $F \in S^1([0,1))$,
$F\leq\alpha/x+\beta|\log x|$
for some $\beta$, $\supp F\subset [0,1/2)$,
$\psi_F=e^F\psi$,
\begin{equation*}
 P\equiv P(F):=e^F(H-\lambda)e^{-F}=H(F)-\lambda,\ H(F)=H+e^F[H,e^{-F}],
\end{equation*}
we have
%
\begin{equation}\label{eq:FH-1}
P(F)\psi_F=0,
\end{equation}
\begin{equation}
 P(F)   =
   H-2(x^2D_x F)(x^2D_x)+(x^2D_x F)^2%
   -\lambda+x^\delta R_1
 \,,\quad R_1\in
\Diffscc^2(\Xb),
\end{equation}
with
\begin{equation}
 \label{eq:FH-1b}
  \re P(F)=H+(x^2D_x F)^2-\lambda+x^\delta R_2,\ \im P(F)=2(x^2\pa_x
 F)(x^2D_x)+x^\delta R_3
 \,,
\end{equation}
$R_2,R_3\in\Diffscc^2(\Xb)$, $R_j$ bounded as long as
$x^2\pa_x F$ is
bounded in $S^0([0,1))$, hence as long as
$F$ is bounded in $S^1([0,1))$.
The coefficients of the terms $x^\delta R_2$, $x^\delta R_3$ are in fact polynomials with vanishing
constant term, in $(x^2\pa_x)^{m+1} F$, $0\le m\leq 1$.

Furthermore,
\begin{equation}
  \label{eq:FH-3}
 \sqi [\re P(F),\im P(F)]=
  \sqi [H+(x^2D_x F)^2,2(x^2\pa_x F)(x^2D_x)]
 + x^{1+ \rdelta}R_4,
\end{equation}
where $R_4\in \Diffscc^2(\Xb)$ is bounded as long as $x^2\pa_x F$
is bounded
in $S^0([0,1))$.%
\end{lemma}

\begin{rem}\label{rem:FH}
{\rm
The presence of bundles $E$
leaves \eqref{eq:FH-1b} unaffected, hence
\eqref{eq:FH-3} holds as well.
}
\end{rem}

\begin{proof}
First note that
\begin{equation}\label{eq:FH-p1}
[x^2D_x, e^{-F}]=-(x^2 D_x F)e^{-F}
 \,,\qquad x^2 D_x F\in S^0([0,1))
 \,,
\end{equation}
so
$e^F[H,e^{-F}]\in\Diffscc^1(\Xb)$ and
indeed expanding $H$ in terms of $(x^2D_x)^2$, $(x^2D_x)(xD_{y_j})$,
etc., we have
\begin{equation}\label{17XII17.5}
 e^F[H,e^{-F}]=-(x^2D_x F) B_1-B_2(x^2 D_x F)
 \,,
\end{equation}
with
$$
B_j-(x^2D_x)\in x^{\rdelta}\Diffscc^1(\Xb).
$$
The dependence of the terms of
$P(F)$ on $F$ thus comes from $x^2 D_x F$, and its commutators via
commuting it through other vector fields (as in rewriting
$(x^2D_x)(x^2 D_x F)$ as $(x^2 D_x F)(x^2D_x)$ plus a commutator term), hence
through $(x^2 D_x)^{m+1} F$, $0\le m\le 1$.
Notice that  writing $H$ as
$$
(x^2 D_x)
a_{00}(x^2D_x)+\sum_j\Big((x^2D_x)a_{0j}
(xD_{y_j})+(xD_{y_j})a_{0j}(x^2D_x)\Big)
$$
modulo terms without factors of $x^2D_x$
and modulo
$x\Diffscc^1(\Xb)$,
where all commutator terms end up after the rearrangements, we find
$$
B_1=a_{00}(x^2D_x)+a_{0j}(xD_{y_j})
 +xR'_1
 \,,
$$
with $R'_1$ bounded in $\Diffscc^1(\Xb)$, and thus with
$a_{00},a_{0j}$ being real, with
\begin{equation}\label{17XII17.6}
B_1-B_1^*\in x \Diffscc^1(\Xb)
\,,
\end{equation}
and similarly for $B_2$.

We use
%
\begin{equation}\begin{aligned}\label{eq:imP-formula}
\im
P(F)=&\frac{1}{2i}(P(F)-P(F)^*)=\frac{1}{2i}(e^F[H,e^{-F}]+[H,e^{-F}]e^F)\\
&=2\Big((x^2
\pa_x F)\tilde B+\tilde B^*(x^2\pa_x F)\Big),\\
&\tilde B-x^2D_x\in x^{\rdelta}\Diffscc^1(\Xb),\ \tilde B-\tilde
B^*\in x \Diffscc^1(\Xb),
\end{aligned}\end{equation}
with $\tilde B=\half(B_1+B_2^*)$,
and
\begin{equation*}\begin{split}
\re P(F)=\half(P(F)+P(F)^*)&=H-\lambda+\half (e^F[H,e^{-F}]-[H,e^{-F}]e^F)\\
&=H-\lambda+\half[e^F,[H,e^{-F}]]
\end{split}\end{equation*}
to prove \eqref{eq:FH-1b} (note that only the $(x^2D_x)^2$ terms in $H$ gives
a non-vanishing contribution to the double commutator).

To prove \eqref{eq:FH-3}, set
\begin{eqnarray*}
    Q & = & \re P(F)-\underbrace{(H+(x^2D_x F)^2-\lambda)}_{=:Q_1}\in
        x^\rdelta\Diffsc^2(\Xb)
         \,,
\\
          Q'
           & = & \im P(F)-\underbrace{2(x^2\pa_x F)(x^2D_x)}_{ =: Q_1'} \in
        x^\rdelta\Diffsc^2(\Xb)
           \,.
\end{eqnarray*}
We can write
\begin{eqnarray*}
  [\re P(F),\im P(F)]  &\equiv&  [Q_1+Q,Q_1'+Q'] \\
    &=& [Q_1,Q_1'] +[Q_1 ,Q']
  + [Q,Q_1' + Q']
   \,.
\end{eqnarray*}
Equation~\eqref{13XII17.11} shows  that both last terms can be put in $R_4$.
\qed\medskip\end{proof}
%
%

In light of \eqref{eq:FH-3}, we need a positivity result for $\sqi [x^2D_x,H]$.
Such a result follows directly from a Poisson bracket computation.
Let $\chi\in\Cinf_c([0,1))$ be supported near $0$, identically $1$ on
a smaller neighborhood of $0$, and let
\begin{equation}
 \label{4XII17.1}
B=\half(\chi(x) x^2D_x+(\chi(x) x^2D_x)^*)
\end{equation}
be the symmetrization of the radial vector field. Here the formal adjoint is taken with respect to the metric measure on $\red{M}$, which deserves some comments:
Indeed, $x^2D_x$ is formally self-adjoint with respect to the measure
$\frac{dx\,dy}{x^2}$,
and if $C$ is formally
self-adjoint with respect to a density $dg'$ then its adjoint
with respect to $\alpha\,dg'$, $\alpha$ smooth real-valued,
is $\alpha^{-1}C\alpha=C+\alpha^{-1}[C,\alpha]$.
In the notation of \eqref{eq:dg-form}, using
$xD_x (x^{-n+1}\tilde g)\in x^{-n+1}(\CI(\Xb)+S^{-\delta}(\Xb))$,
we find
\begin{equation*}
ib:=x^{-1}(B-\chi(x) x^2D_x)\in \Cinf(\Xb) +S^{-\delta}(\Xb)
\end{equation*}
and $b$ is real-valued. It is easy to check that
$b|_{\pa\red{M}}=\frac{n-1}{2}$,
%
%
where $n=\dim \red{M}$.

For the next theorem we also introduce
\begin{equation}\label{4XII17.2}
A=\half(\chi(x) xD_x+(\chi(x) xD_x)^*)
 \,.
\end{equation}

Recall from \cite[Equation~(4.2)]{RBMSpec} that the space $\Psiscc^{s,r}(\Xb)$ of scattering pseudo-differential operators
is locally defined using quantisations of product-type symbols satisfying estimates
\begin{equation}\label{3XII17.1}
 | \partial^\alpha_z \partial^\beta_\zeta a(z,\zeta) |\le
  C(\alpha,\beta) (1+|z|)^{r-|\alpha|} (1+|\zeta|)^{s-|\beta|}
  \,,
\end{equation}
where $z$ is thought as a local coordinate on the manifold and $\zeta$
the momentum variable. {\em Notice that we have the negative of the
  convention of \cite{RBMSpec} for the growth order $r$, so for us
  $r>0$ means growing coefficients.} (Parenti
\cite{Parenti:Operatori} and
Shubin~\cite{Shubin:Pseudodifferential} introduced this class earlier
on $\RR^n$.)  Here `locally' means a coordinate
identification of an open set on the manifold with boundary with an
open set on the radial compactification $\overline{\RR^n}$ of $\RR^n$
discussed at
the end of Section~\ref{s25X17.1}.
Near $\pa\overline{\RR^n}$, since in that region $\overline{\RR^n}$ is identified with
$[0,\ep)_x\times\sphere^{n-1}$, the coordinate identification is thus with the
closure of an asymptotically conic
subset of $\RR^n$.

On the other hand, $\Psisc^{s,r}(\Xb)$ is the subspace of
$\Psiscc^{s,r}(\Xb)$ consisting of classical
pseudodifferential operators, see
\cite[Equation~(4.7)]{RBMSpec}. Here classical means, in terms of the
local description above, that the symbol
$a$ has a one-step (i.e.,  with powers in the expansions stepping by one)
 polyhomogeneous expansion both in terms of the
defining function of spatial infinity, $|z|^{-1}$, and the defining
function of momentum infinity, $|\zeta|^{-1}$, cf.\ the discussion at
the end of Section~\ref{s25X17.1}. This joint behavior can again be
encoded via a compactification. We compactify the momentum
variable $\zeta$ similarly to the position variable $z$, so the amplitude $a$ is considered as a function
on the interior $\RR^n_z\times\RR^n_\zeta$ of $\overline{\RR^n}\times
\overline{\RR^n}$. For $s=r=0$, classicality means that $a$ extends
(necessarily uniquely) as a smooth function to $\overline{\RR^n}\times
\overline{\RR^n}$; for general $s$, $r$, the statement is the
analogous extendability for
$$
(1+|z|^2)^{-r/2}(1+|\zeta|^2)^{-s/2} a;
$$
the expression $(1+|\cdot|^2)^{1/2}$ is used in place of $1+|\cdot|$ to ensure smoothness at
the origin.

Taking into account that differential operators have amplitudes that
are polynomial, thus classical, in the momentum variable, these
definitions are consistent with those of
$\Diffsc^{s,-r}(\Xb)\subset\Psisc^{s,r}(\Xb)$ (notice the change in
sign!) and
$S^r\Diffsc^s(\Xb)=\Diffscc^{s,-r}(\Xb)\subset\Psiscc^{s,r}(\Xb)$, for $s$ a non-negative integer.

\begin{prop}\label{prop:Mourre-est}
Let $A$ and $B$ be given by \eqref{4XII17.1}-\eqref{4XII17.2}.
There exist $R \in x \Psisc^{0,0}(\Xb)$ and $ K\in x^{1+\rdelta}\Psiscc^{2,0}(\Xb)$
such that
\begin{equation}\label{eq:Mourre-ch}
\sqi [B,H]= 2\lambda x-
 2
 BxB+(H-\lambda)R+R^*(H-\lambda)+ K.
\end{equation}
In addition, there exist
$\tilde R\in \Psisc^{0,0}(\Xb)$,
$\tilde K  \in x ^\rdelta \Psiscc^{2,0}(\Xb)$, such that
%
\begin{equation}\label{eq:Mourre-0}
\sqi [A,H] = 2\lambda+(H-\lambda)\tilde R+\tilde R^*(H-\lambda)+ \tilde K.
\end{equation}
\end{prop}

\begin{rem}\label{rem:Mourre}
{\rm
Again, the presence of bundles makes no difference in this proof.
}
\end{rem}

\begin{proof}
We show \eqref{eq:Mourre-ch}; the proof of \eqref{eq:Mourre-0} is
entirely analogous.

Adding terms to $H$ which differ from it by an element of
$S^{-\delta}\Diffsc^2(\Xb)=\Diffscc^{2,\delta}(\Xb)$, results in a term in the commutator
in $ x^{1+\delta}\Diffsc^2(\Xb)
 $ that
can be absorbed into $K$, and similarly for adding a term in
$\Diffscc^{1,\delta}(\Xb)$ to $B$. Thus, all of the terms arising from
$S^{-\delta}$ terms in either $g$ or $V$ can be ignored. A
straightforward principal symbol computation (which recall is modulo
one order gain {both} in derivatives and   decay) then gives that (modulo
these $S^{-\delta}$ terms which give $K$-absorbable results) the
principal symbol of the left hand side is the same as that of
$xH+Hx-
2
BxB$,
so that of $2x\lambda-
2
 BxB+x(H-\lambda)+(H-\lambda)x$, proving the
proposition since the agreement of the principal symbols means that
the operators agree modulo $x^2\Diffsc^{2,0}(\Xb)$, which again is
absorbable into $K$.
\qed\medskip\end{proof}

We also need a somewhat more general computation. For this, let $k>0$,
and we note
first that, with $t\in(0,1]$ a parameter, $(1+t/x)^{-k}=(x/(x+t))^k$, as a function of $x$, is a
symbol of order $-k$ for $t>0$, and is uniformly bounded in symbols of
order $0$. Indeed, this follows from
%
%
$$
x\pa_x (1+t/x)^{-k}=k(t/x)(1+t/x)^{-k-1}=k(1+t/x)^{-k}-k (1+t/x)^{-k-1},
$$
so iterative regularity under derivatives is easy to check, and from
$(1+t/x)^{-k}\leq 1$ and for $t>0$, $(1+t/x)^{-k}\leq t^{-k}
x^k$. Thus, $(1+t/x)^{-k}x^2 D_x$ is in $S^{-k}\Diffscc^1(\Xb)$ for
$t>0$, and is uniformly bounded in $S^{0}\Diffscc^1(\Xb)$, converging
to $x^2 D_x$ in $S^{\delta'}\Diffscc^1(\Xb)$ for $\delta'>0$. This
proves the first part of the following proposition.

\begin{prop}\label{prop:poly-weight}
For $t\in[0,1]$, let
$$
B_{s,k,t}=\half((\chi(x) x^{-s}(1+t/x)^{-k}x^2 D_x)+(\chi(x) x^{-s}(1+t/x)^{-k}x^2 D_x)^*).
$$
Then for $t>0$, $B_{s,k,t}\in S^{s-k}\Diffsc^1(\Xb)$, and
$\{B_{s,k,t}:\ t\in[0,1]\}$ is uniformly bounded in
$S^s\Diffsc^1(\Xb)$, with $B_{s,k,t}\to B_{s,k,0}$ in
$S^{s+\delta'}\Psiscc^{1,0}(\Xb)$, $\delta'>0$ arbitrary, as $t\to 0$.

Furthermore, with $\hat R_{s,k,t}$ uniformly bounded in
$S^{s-1}\Diffsc^{0,0}(\Xb)$, $\hat K_{s,k,t}, \hat K'_{s,k,t}$ uniformly bounded in
$
{S^{s-1-\delta}\Diffsc^{2,0}(\Xb)}$,
%
%
\begin{equation}
 \label{10XII17.1}
 \begin{aligned}
    \sqi [B_{s,k,t},H]&=2x^{1-s}(1+t/x)^{-k}\Big(\Big(s-k\frac{t/x}{1+t/x}\Big)(x^2D_x)^2+\Delta_h\Big)+\hat
    K'_{s,k,t}
\\
    &=2x^{1-s}(1+t/x)^{-k}\lambda+2(x^2D_x)^*x^{1-s}(1+t/x)^{-k}
     \Big(s-1-k\frac{t/x}{1+t/x}\Big)
    (x^2D_x)
\\
    &\qquad\qquad+(H-\lambda)\hat R_{s,k,t}+\hat R_{s,k,t}^*(H-\lambda)+\hat K_{s,k,t}
     \,.
 \end{aligned}
\end{equation}
Thus, if $s-k\geq 1$,
\begin{equation*}\begin{aligned}
\sqi [B_{s,k,t},H]&\geq 2x^{1-s}(1+t/x)^{-k}\lambda+(H-\lambda)\hat R_{s,k,t}+\hat R_{s,k,t}^*(H-\lambda)+\hat K_{s,k,t}.
\end{aligned}\end{equation*}
\end{prop}

\begin{rem}
{\rm
The main point of this computation is that on the one hand
$s-k\frac{t}{x}\big(1+\frac{t}{x}\big)^{-1}\geq s-k$, thus is positive
with a positive lower bound if
$k<s$, on the other hand converges to $s$ in any growing symbol space,
$S^{\delta'}$, $\delta'>0$, as $t\to 0$.
}
\end{rem}

\begin{proof}
It only remains to do the commutator calculation.
Again, this is a principal symbol computation in which all of the terms arising from
$S^{-\delta}$ terms in either $g$ or $V$ can be ignored.
\qed\medskip\end{proof}
%
%

\subsection{Super-polynomial decay}\label{sec:poly}
As a starting point, we show that $L^2$ solutions decay
superpolynomially. From a microlocal analysis perspective, this
follows from propagation results in the framework of scattering
pseudodifferential operators as in \cite{RBMSpec}, cf.\ also the
arguments preceding \cite[Proposition~4.13]{Vasy:Minicourse}, but we give a more
elementary argument (under, however, stronger assumptions than those of
\cite{RBMSpec}).

We recall  the scattering Sobolev spaces
$\Hsc^{s,r}(\Xb)=x^{r}\Hsc^{s,0}(\Xb)$, which are modelled, via local
coordinate identification much as for the scattering
pseudodifferential operators, on the
standard weighted Sobolev spaces
$$
H^{s,r}(\RR^n)=\langle z\rangle^{-r} H^s(\RR^n)=\{u\in\cS'(\RR^n):\
\langle z\rangle^{r} u\in H^s(\RR^n)\}
\,,
$$
where
 $\langle z\rangle = (1+|z|^2)^{1/2}$.
Correspondingly, $\Hsc^{s,r}(\Xb)$ is the space of tempered
distributions $u$ on $\Xb$ (dual of $\dCI(\Xb)$, consisting of $\CI$
functions vanishing with all derivatives at $\pa \red{M}$) such that for all
$A\in\Psiscc^{s,r}(\Xb)$, $Au\in L^2_\scl(\Xb)$. This in turn is
equivalent to requiring $Au\in L^2_\scl(\Xb)$ for a {\em single fully
  elliptic} $A\in\Psiscc^{s,r}(\Xb)$; full ellipticity means that for the local coordinate
amplitude $a$, $|a|$ has a lower bound $c(1+|z|)^r(1+|\zeta|)^s$, $c>0$,
for $|z|+|\zeta|$ sufficiently large. For instance, with $\Delta$ the
positive Laplacian of a scattering metric, $(1+\Delta)^{s/2}$ is
elliptic in $\Psiscc^{s,0}(\Xb)$, and thus $x^{-r}(1+\Delta)^{s/2}$ is
elliptic in $\Psiscc^{s,r}(\Xb)$. Elliptic regularity in the
differential order sense, i.e.\ for an operator $A$ with principal
symbol satisfying $|a(z,\zeta)|\geq c(1+|z|)^r(1+|\zeta|)^s$, $c>0$,
for $|\zeta|$ sufficiently large (but not necessarily if $|z|$ is
large and $\zeta$ is in a bounded region), is the statement that
$Au\in\Hsc^{k-s,p-r}(\Xb)$ and $u\in \Hsc^{k',p}$ for some $k'$
implies $u\in\Hsc^{k,p}(\Xb)$, with a corresponding estimate on
$u$. In particular, there is a constant $C=C_{k,k',p}$ such that if $Au=0$ then $\|u\|_{\Hsc^{k,p}(\Xb)}\leq
C_{k,k',p}\|u\|_{\Hsc^{k',p}(\Xb)}$; note that there is an improvement
in the differentiable, but not in the decay order of the space.

As a first step it is useful to observe the following:

\begin{lemma}\label{lemma:sc-formal-adjoint-works}
Suppose that $Q\in S^{1+2\beta}\Psisc^{1,0}(\Xb)$ and $(H-\lambda)\psi=0$, $x^{-\beta}\psi\in
L^2_\scl(\red{M})$. Then
%
$$
( [Q,H]\psi,\psi)=0.
$$
\end{lemma}

\begin{proof}
As $(H-\lambda)\psi=0$, by elliptic regularity
\begin{equation}\label{eq:elliptic-poly}
\|\psi\|_{x^p\Hsc^k(\Xb)}\leq \tilde b_{1,k,p}\|x^p\psi\|,
\end{equation}
Thus, by elliptic regularity, the differential order of operators below
never matters.

Formally the lemma follows from
$$
( [Q,H]\psi,\psi)=( [Q,H-\lambda]\psi,\psi)= (
(H-\lambda)\psi,Q^*\psi)+( Q\psi,(H-\lambda)\psi)=0,
$$
but care needs to be taken as the pairing on left hand side is {\em
  only} defined (as the dual pairing between $x^{-\beta}L^2_\scl(\Xb)$
and $x^\beta L^2_\scl(\Xb)$) due to the fact that
$[Q,H]\in S^{2\beta}\Psisc^{2,0}(\Xb)$, while
$$
 Q(H-\lambda)\,,
  \quad
   (H-\lambda)Q\in
S^{2\beta+1}\Psisc^{3,0}(\Xb)
 \,,
$$
so the pairing $(
(H-\lambda)Q\psi,\psi)$ is not a priori defined, with the lack
of sufficient decay being the issue.

To remedy this, one
simply regularizes; here we use $(1+t/x)^{-1}=\frac{x}{x+t}$, $t>0$, as a
regularizer since it also plays a role below, this is uniformly
bounded as $t\to 0$ (by $1$), and is $O(x)$ for $t>0$, removing the
pairing issue. Namely, we have, with the pairing being the dual pairing,
%
%
\begin{equation}
 \label{11XII17.1}
 \begin{aligned}
( [Q,H]\psi,\psi)&
 =
  \lim_{t\to 0}( (1+t/x)^{-1}
    [Q,H]\psi,\psi)
\\
&=\lim_{t\to 0}( (1+t/x)^{-1}
(Q(H-\lambda)-(H-\lambda)Q)\psi,\psi)\\
&=\lim_{t\to 0}\Big(( (1+t/x)^{-1}
Q(H-\lambda)\psi,\psi)-( (1+t/x)^{-1}
(H-\lambda)Q\psi,\psi)\Big)
\\
 &
  =
    -\lim_{t\to 0}\big( \big([(1+t/x)^{-1},
    H-\lambda]+(H-\lambda) (1+t/x)^{-1}\big)Q\psi,\psi\big)
\\
 &=-\lim_{t\to 0}\Big(( [(1+t/x)^{-1},
H-\lambda]Q\psi,\psi)+( (H-\lambda) (1+t/x)^{-1}Q\psi,\psi)\Big),
 \end{aligned}
\end{equation}
where the penultimate equality used that
$(H-\lambda)\psi=0$.

Now, $[(1+t/x)^{-1},(H-\lambda)]$ is uniformly bounded (as $t\to 0$)
in $S^{-1}\Diffsc^{1,0}(\Xb)$, and indeed converges to $0$ in
$S^{-1+\delta'}\Diffsc^{1,0}(\Xb)$, $\delta'>0$, so
$[(1+t/x)^{-1},(H-\lambda)]Q$ is uniformly bounded in
$S^0\Diffsc^{2,0}(\Xb)$ and converges to $0$ in
$S^{\delta'}\Diffsc^{2,0}(\Xb)$.
This implies that
$[(1+t/x)^{-1},(H-\lambda)]Q\psi$ converges to $0$ in
$\Hsc^{-1,0}(\Xb)$  {as $t\to 0$}, as can be seen that if $\psi_j\to\psi$ in
$L^2_\scl(\Xb)$ with $\psi_j\in\dCI(\Xb)$, then
$[(1+t/x)^{-1},(H-\lambda)]Q\psi$ converges to $0$ as $t\to 0$ in $\dCI(\Xb)$
thus
in $\Hsc^{-1,0}(\Xb)$, while given any $\ep>0$,
$[(1+t/x)^{-1},(H-\lambda)]Q(\psi-\psi_j)$ is $<\ep$ in
$\Hsc^{-1,0}(\Xb)$ for sufficiently large $j$. Thus,
\begin{equation*}\begin{aligned}
( [Q,H]\psi,\psi)&=-\lim_{t\to 0}\Big(( [(1+t/x)^{-1},
H-\lambda]Q\psi,\psi)+( (H-\lambda)
(1+t/x)^{-1}Q\psi,\psi)\Big)\\
&=-\lim_{t\to 0}( (H-\lambda)
(1+t/x)^{-1}Q\psi,\psi)\\
&=-\lim_{t\to 0}(
(1+t/x)^{-1}Q\psi, (H-\lambda)\psi)=0,
\end{aligned}\end{equation*}
proving the lemma.
%
\qed\medskip\end{proof}

\begin{prop}\label{prop:poly-decay}
Let $\lambda>0$, and suppose that
$\psi\in L^2_{\scl}(\Xbext)$ satisfies $H\psi=\lambda\psi$. Then for all
$\beta\in\Real$, $x^{-\beta}\psi
\in L^2_{\scl}(\Xbext)$.
\end{prop}

\begin{rem}
 \label{R23XII17.1}
{\rm
In the proof we will assume for clarity that $(\red{M},g)$ is complete with  $\psi\in L^2_{\scl}(\Xb)$.
If we merely assume $x^{-\beta}\psi\in
L^2_\scl(\Xext)$, then all integrations by parts, such as e.g.\ those involved  in the proof of Lemma~\ref{lemma:sc-formal-adjoint-works}, should be carried-out on $\Xext$. This will introduce  controlled  boundary terms at the inner boundary $\{x=x_1\}$ which will not affect the argument. Equivalently, in the proof the eigenfunction $\psi$ can be multiplied by a cut-off function which vanishes near $\{x=x_1\}$ and equals one near $\{x=0\}$, leading to error terms in the equations which can be estimated by $C \|\psi\|$, and resulting in the same conclusions.
}
\end{rem}


\begin{proof}
Let
$$
\beta_1=\sup\{\beta\in[0,\infty):\ x^{-\beta} \psi\in L^2_\scl(\Xb)\},
$$
and assume $\beta_1$ is finite. Let
$\beta\in(\beta_1,\beta_1+\delta/2)>0$.
Take $s=1+2\beta$ and $0<k\leq \min(1,\delta) {=\delta}$
 with $\beta-k/2\in[0,\beta_1)$ if
$\beta_1>0$, $\beta-k/2=0$ if $\beta_1=0$. Thus,
$s-k=1+2(\beta-k/2)\geq 1$, and $\frac{s-1-k}{2}=\beta-k/2$ is either
$0$ or $<\beta_1$, so in either case
$$
x^{-(s-k-1)/2}\psi\in L^2_\scl(\red{M}),
$$
and $\frac{s-1-\delta}{2}\leq \frac{s-1-k}{2}$ as well, so
$$
x^{-(s-1-\delta)/2}\psi\in L^2_\scl(\red{M}).
$$

Apply
Proposition~\ref{prop:poly-weight} with this $s,k$, using that $\hat
K_{s,k,t}$ uniformly bounded in $S^{s-1-\delta}\Diffsc^{2,0}(\Xb)$, so
in view of elliptic regularity, \eqref{eq:elliptic-poly},
$$
|(\hat K_{s,k,t}\psi,\psi)|\leq \tilde C_s \|x^{-(s-1-\delta)/2}\psi\|^2,
$$
and
$\hat R_{s,k,t}$ is in
$S^{s-k-1}\Diffsc^{0,0}(\Xb)$ for $t>0$ (not uniformly bounded in $t$ though!),
to conclude, using $(H-\lambda)\psi=0$, that
\begin{equation*}\begin{aligned}
( \sqi [B_{s,k,t},H]\psi,\psi)&\geq 2 \lambda
\|(1+t/x)^{-k/2}x^{-\beta}\psi\|^2+( (H-\lambda)\hat
R_{s,k,t}\psi,\psi)\\
&\qquad\qquad+(\hat
R_{s,k,t}^*(H-\lambda)\psi,\psi)+(\hat K_{s,k,t}\psi,\psi)\\
&\geq 2 \lambda
\|(1+t/x)^{-k/2}x^{-\beta}\psi\|^2-\tilde C_{s}\|x^{-(s-1-\delta)/2}\psi\|^2.
\end{aligned}\end{equation*}
On the other hand, by Lemma~\ref{lemma:sc-formal-adjoint-works},
$$
( \sqi [B_{s,k,t},H]\psi,\psi)=0.
$$
Thus, we conclude that
$$
2 \lambda
\|(1+t/x)^{-k/2}x^{-\beta}\psi\|^2
$$
is uniformly bounded as $t\to 0$, thus $x^{-\beta}\psi\in L^2$,
contradicting the definition of $\beta$.

This shows that $\beta_1$ is not finite, hence proves the proposition.
\qed\medskip\end{proof}

\subsection{Superexponential decay}\label{sec:exp}

Using the global positive commutator estimate,
Proposition~\ref{prop:Mourre-est}, we can now prove a decay faster-than-any-exponential
of non-threshold eigenfunctions. For this part of the paper,
we could adapt the proof of
Froese and Herbst \cite{FroExp} in Euclidean potential scattering,
as was done in \cite{Vasy:Propagation-2} in the geometric potential
scattering setting. However, we focus on the approach
that will play a crucial role in the proof of unique continuation at
infinity. Nonetheless, a modification of the Froese-Herbst commutator
will play a role when $\alpha=0$ (in the notation of Lemma~\ref{lemma:FH-1}),
where
conjugated Hamiltonian is close to being self-adjoint (in fact, it
is, if $F=0$), so we will use a modification of $xD_x$, more precisely
a rescaling of $\im P$, for a commutator estimate
in place of $\im P$.

\begin{thm}\label{thm:exp-decay}
Let $\lambda>0$, and suppose that
$\psi\in L^2_{\scl}(\Xbext)$ satisfies $H\psi=\lambda\psi$. Then for all
$\alpha\in\Real$, $e^{\alpha/x}\psi
\in L^2_{\scl}(\Xbext)$.
\end{thm}

\begin{proof}
We start by pointing-out that Remark~\ref{R23XII17.1} applies again.

The proof is by contradiction.
First note that
$$
 \psi\in\dCinf(\Xb)
$$
by Proposition~\ref{prop:poly-decay} and the usual weighted elliptic estimates.
(Recall that
$\dCinf(\Xb)$ denotes the  space of smooth functions which decay to infinite order
 at the boundary.)

 Let
\begin{equation*}
\alpha_1=\sup\{\alpha\in[0,\infty):\ \exp(\alpha/x)\psi\in L^2_{\scl}(\Xb)\}.
\end{equation*}
If $\alpha_1=0$, then let $\alpha=0$,
otherwise suppose that $\alpha<\alpha_1$,
and in either case
 $\gamma\in(0,1]$ will be a constant satisfying $\alpha+\gamma>\alpha_1$. These two cases will require separate treatment.

The $\alpha_1>0$ case is more representative of the proof of the
unique continuation result, so we will start with that.  We show in
this case that for sufficiently small $\gamma$
(depending only on $\alpha_1$) $\exp((\alpha+\gamma)/x)\psi\in L^2_{\scl}(\Xb)$,
which contradicts our assumption on $\alpha_1$ if $\alpha$ is close enough
to $\alpha_1$.

We start with a general discussion, so we do not yet make assumptions
on $\alpha_1$.

Below we use two positivity estimates, namely \eqref{eq:FH-3} and the
Mourre-type
estimate, Proposition~\ref{prop:Mourre-est},
at energy $\lambda+\alpha_1^2$ (i.e.\ with $\lambda$ replaced by this
in the statement of the proposition), with $B=\chi(x)x^2D_x+(\chi(x)x^2D_x)^*$.
That is, since
$\lambda+\alpha_1^2>0$,
there exists $c_0>0$,
$R\in\Psisc^{0,0}(\Xb)$, $K\in\Psisc^{2,0}(\Xb)$, such that
for $\psit\in L^2_{\scl}(\Xb)$,
\begin{equation}\begin{split}\label{eq:pos-comm-88}
(\psit,&\sqi [B,H]\psit)\\
&\geq c_0\|x^{1/2}\psit\|^2-
 \red{2} \re(\psit,x(x^2D_x)^2\psit)\\
&\qquad+\re((H-\lambda-\alpha_1^2)\psit,xR\psit)
+\re(x^{( {1+}\delta)/2}\psit,Kx^{( {1+}\delta)/2}\psit).
\end{split}\end{equation}
We will apply this  with $\psit=\psi_F$, with $F$ given by \eqref{11XII17.2} below.

We first note that we certainly have for all $\beta\in\Real$,
$\exp(\alpha/x)x^\beta\psi\in L^2_{\scl}(\Xb)$, due to our choice of $\alpha$.
We apply  Lemma~\ref{lemma:FH-1} with $\beta\ge 1$ and with
\begin{equation}
 \label{11XII17.2}
    F\equiv F_\beta:=\frac{\alpha}{x}+\beta\log(1+\frac{\gamma}{\beta x}),
\end{equation}
and let
$$
 \psi_\beta:=\psi_F\equiv e^F\psi
  \,.
$$
The reason for the choice \eqref{11XII17.2} is that
on the one hand
$F(x)\to (\alpha+\gamma)/x$ as $\beta\to\infty$, so in the limit we will
obtain an estimate on $e^{(\alpha+\gamma)/x}\psi$, and on the other hand
$F(x)\leq \frac{\alpha}{x}+\beta|\log x|$,
 so $e^{F_\beta}$ is bounded
by $x^\beta e^{\alpha/x}$, for all values of $\beta$, i.e.\ $e^{F_\beta}$
provides a `regularization' (in terms of growth) of $e^{(\alpha+\gamma)/x}$,
so that Lemma~\ref{lemma:FH-1} can be applied.

Note that $F=F_\beta\in S^1([0,1))$, and
$F_\beta$ is uniformly bounded in $S^1(
[0,1))$
for $\beta\in[1,\infty)$, $\alpha\in[0,\alpha_1)$ (or $\alpha=\alpha_1$
if $\alpha_1=0$), $\gamma\in[0,1]$. Indeed,
\begin{equation*}
0\leq -x^2\pa_x F=\alpha+\gamma(1+\frac{\gamma}{\beta x})^{-1}\leq\alpha+\gamma,
\end{equation*}
and in general $(x\pa_x)^m(1+\frac{\gamma}{\beta x})^{-1}
=(-r\pa_r)^m(1+r)^{-1}$, $r=\frac{\gamma}{\beta x}$, so the uniform boundedness
of $F$ follows from $(1+r)^{-1}$ being a symbol in the usual sense on
$[0,\infty)$. In particular, all symbol norms of $-x^2\pa_x F-\alpha$ are
$O(\gamma)$.

Below, when $\alpha=0$, we will need to consider
$(-x^2\pa_x F)^{-1}(x^2\pa_x)^m (-x^2\pa_x F)$, $m\geq 0$. By Leibniz' rule,
this can be written as
$\sum_{j\leq m}c_j x^{m} (-x^2\pa_x F)^{-1}(x\pa_x)^j (-x^2\pa_x F)$.
In terms of $r$, $(-x^2\pa_x F)^{-1}(x\pa_x)^j (-x^2\pa_x F)$ {with
  $\alpha=0$} takes the form
$$
 (1+r)(-r\pa_r)^j(1+r)^{-1}
 \,,
$$
hence it
is still bounded on $[0,\infty)$, so in fact
\begin{equation}\label{eq:factor-8}
x^{-m}(-x^2\pa_x F)^{-1}(x^2\pa_x)^m (-x^2\pa_x F),\quad m\geq 0,
\end{equation}
is uniformly bounded on $[0,\infty)$. In fact, \eqref{eq:factor-8} is uniformly
bounded in $S^0([0,1))$, since applying $x\pa_x$ to it gives rise
to additional factors such as
\begin{equation*}
(-x^2\pa_x F)^{-k}(x\pa_x)^k(-x^2\pa_x F),
\end{equation*}
which  {we have just seen to be uniformly bounded on $[0,\infty)$}.

We remark first that $P(F)\psi_F=0$, so by elliptic regularity,
\begin{equation*}
\|\psi_F\|_{x^p\Hsc^k(\Xb)}\leq b_{1,k,p}\|x^p\psi_F\|,
\end{equation*}
with $b_{1,k,p}$ independent of $F$ as long as $\alpha$ is bounded.
{This follows from the fact that the estimates on the derivatives of $F$, as needed for controlling $b_{1,k,p}$,
are uniform in  $\alpha\in [0,\alpha_1]$,  $\gamma\in [0,1]$ and $\gamma/\beta\in[0,1]$.}
{\em In general,
below $b_j$ denote positive
constants that are independent of $\alpha,\beta,\gamma$
in these intervals, and $R_j$ denote operators which are uniformly
bounded in $\Diffscc^2(\Xb)$, or on occasion in $\Psiscc^{m,0}
(\Xb)$, for some $m$.}
(Note {again} that, by elliptic regularity, the differential order
never matters.)

{As already pointed out, the} proof is slightly different in the cases $\alpha>0$ and $\alpha=0$
since in the latter case the usually dominating term, $-2\alpha x^2D_x$,
of $\im P$ vanishes.

When $\alpha>0$, the key step in the proof of this theorem arises from considering,
with $P\equiv P_\beta:=H(F)-\lambda$,
\begin{equation*}
P^*P=(\re P)^2+(\im P)^2+i(\re P\im P-\im P\re P),
\end{equation*}
so
\begin{equation}\label{eq:comm-8}
0=(\psi_F,P^*P\psi_F)=\|\re P\psi_F\|^2+\|\im P\psi_F\|^2+
(\psi_F,\sqi [\re P,\im P]\psi_F).
\end{equation}
The first two terms on the right-hand side are non-negative, so the
key issue is the positivity of the commutator. Note that
%
\begin{equation}\begin{split}\label{eq:comm-9}
\re P=H-\alpha^2-\lambda+\gamma R_1+x^{\rdelta}R_2,\\
\im P=-2\alpha x^2D_x+\gamma R_3+x^{\rdelta}R_4.
\end{split}\end{equation}
%
By \eqref{eq:FH-3},
%
\begin{equation}
 \begin{split}
    \sqi [\re P,\im P]=&2\alpha \sqi [x^2D_x,H]+x \gamma R_5
    +x^{1+\rdelta} R_{6}
\,.
\end{split}
 \label{15XII17.1}
\end{equation}
Hence, from \eqref{eq:comm-8} and
\eqref{eq:pos-comm-88}, with the $\re P$ and $\im P$ terms in the
scalar products arising
from \eqref{eq:pos-comm-88},
\begin{equation}\begin{split}\label{eq:comm-24}
0\geq &\|\re P\psi_F\|^2+\|\im P\psi_F\|^2+2\alpha c_0\|x^{1/2}\psi_F\|^2
 + {\gamma}
 (\psi_F,x    R_{11}\psi_F)\\
 &+(\psi_F,x R_{12}\re P\psi_F)+(\psi_F,\re Px R_{13}\psi_F)\\
&
+ (\psi_F,x   R_{14}\im P\psi_F)+(\psi_F,\im Px   R_{15}\psi_F)
  +(\psi_F,x^{1+\rdelta}R_{16}\psi_F)
  \,.
\end{split}\end{equation}
Now,  terms such as
 $|(\psi_F,x^{1+\rdelta}R_{16}\psi_F)|$
 can be estimated
by $b_2\|x^{(1+ \rdelta )/2}\psi_F\|^2$,
while
 $\gamma|(\psi_F,x    R_{11}\psi_F)|$
may be estimated by $\gamma b_3\|x^{1/2}\psi_F\|^2$,
while
\begin{equation*}\begin{split}
&|(\psi_F,x   R_{12}\re P\psi_F)|\leq b_4\|x \psi _F\|\|\re P\psi_F\|
\leq b_4(\ep^{-1}\|x \psi _F\|^2+\ep\|\re\psi_F\|^2)
\,,
\\
&|(\psi_F,x   R_{14}\im P\psi_F)|\leq b_5\|x \psi _F\|\|\im P\psi_F\|
\leq b_5(\ep^{-1}\|x \psi _F\|^2+\ep\|\im\psi_F\|^2)
 \,,
\end{split}\end{equation*}
with similar estimates for the remaining terms.
Putting this together, \eqref{eq:comm-24} yields
\begin{equation}\begin{split}\label{eq:comm-32}
0\geq &
 (1-b_6 \ep)
  \|\re P\psi_F\|^2
   +
    (1-b_7 \ep)
     \|\im P\psi_F\|^2\\
&+(2\alpha c_0
-\gamma b_8)\|x^{1/2}\psi_F\|^2-b_9(\ep)\|x^{(1+\rdelta)/2} \psi _F\|^2.
\end{split}\end{equation}
For $\deltat>0$, in $x\geq\deltat$, $x|\psi_F|=xe^{F}|\psi|\leq
b_{10}(\deltat)|\psi|$, so
\begin{equation*}\begin{split}
    \|x^{(1+\rdelta)/2} \psi _F\|^2&=\|x^{(1+\rdelta)/2} \psi _F\|^2_{x\leq\deltat}+\|x^{(1+\rdelta)/2} \psi _F\|^2_{x\geq\deltat}
\\
    &
    \leq \deltat^{ \rdelta }\|x^{1/2}\psi_F\|^2_{x\leq\deltat}
    +b_{10}(\deltat)\|\psi\|^2_{x\geq\deltat}
\\
 &
    \leq \deltat^{ \rdelta }\|x^{1/2}\psi_F\|^2
+b_{10}(\deltat)\|\psi\|^2.
\end{split}
\end{equation*}
Thus, \eqref{eq:comm-32} yields that
\begin{equation}
 \begin{split}\label{eq:comm-64}
 0\geq &(1-b_6 \ep)\|\re P\psi_F\|^2+(1-b_7 \ep)\|\im P\psi_F\|^2
\\
 &
 +(2\alpha c_0
    -\gamma b_8-b_9(\ep)\deltat^{ \rdelta })\|x^{1/2}\psi_F\|^2-b_{10}(\deltat)\|\psi\|^2.
 \end{split}
\end{equation}
Hence, choosing $\ep>0$ sufficiently small so that $b_6\ep<1$, $b_7\ep<1$,
then choosing $\gamma_0>0$ sufficiently small so that $b_{11}=2\alpha c_0
-\gamma_0 b_8>0$, we deduce that for $\gamma<\gamma_0$,
\begin{equation}\label{eq:comm-88}
b_{10}(\deltat)\|\psi\|^2\geq (b_{11}-b_9\deltat^{ \rdelta })\|x^{1/2}\psi_F\|^2.
\end{equation}
But, for $\deltat^{ \rdelta }\in(0,\frac{b_{11}}{b_9})$,
this shows that $\|x^{1/2}\psi_F\|^2$
is uniformly bounded as $\beta\to\infty$.
Noting that $F$ is an increasing function of $\beta$
and $\psi_F$ converges to $e^{(\alpha+\gamma)/x}\psi$ pointwise,
we deduce from the monotone convergence theorem that
\begin{equation*}
x^{1/2}e^{(\alpha+\gamma)/x}\psi\in L^2_\scl(\Xb),
\end{equation*}
so for
$\gamma'<\gamma$, $e^{(\alpha+\gamma')/x}\psi\in L^2_\scl(\Xb)$.
%

We pass now to the  case $\alpha=0$. Then \eqref{15XII17.1} becomes
\begin{equation}\label{15XII17.2}
 \sqi [\re P,\im P]=\gamma x R_{16}
 + x^{1+\delta} R_{17}
  \,.
\end{equation}
The calculations so far lead instead to
$$\|\re P\psi_F\|\leq b_{12}\|x^{1/2}\psi_F\|
 \,,
 \quad
\|\im P\psi_F\|\leq b_{12}\|x^{1/2}\psi_F\|
 \,.
$$
This  implies that
\begin{equation}\label{eq:deg-re-est}
 \|(H-\lambda)\psi_F\|\leq \gamma b_{13}\|\psi_F\|
  +b_{14}\|x^{\rdelta }\psi_F\|
 \,,
\end{equation}
\begin{equation}\label{eq:deg-im-est}
\|(x^2\pa_x F)x^2 D_x   \psi _F\|\leq b_{15}\|x^{\rdelta }\psi_F\|
 \,.
\end{equation}

To continue,
instead
of the degenerating
 commutator $[\re P,\im P]$ (where the term $x^2 \partial_x F$ loses
 its leading order contribution from $\alpha$), we recall from
 \eqref{eq:imP-formula} that
\begin{equation*}\begin{aligned}
\im P&=x(x^2\pa_x F) \tilde A+\tilde A^* x(x^2\pa_x F),\\
\tilde
A&=xD_x+x^{-1+\rdelta} \tilde R_1=A+x^{-1+\rdelta}\tilde R_2,\ \tilde
A-\tilde A^*\in\Diffscc^1(\Xb)
\end{aligned}\end{equation*}
with $\tilde R_1,\tilde R_2$ uniformly bounded
in $\Diffscc^1(\Xb)$, and consider $P^*\tilde A-\tilde AP$.
Note that
\begin{equation*}\begin{aligned}
\im P=(\im P)^*&=2x(x^2\pa_x F)\tilde A+[\tilde A, x(x^2\pa_x
F)]+(\tilde A^*-\tilde A) x(x^2\pa_x
F)\\
&=2x(x^2\pa_x F)\tilde A^*+x\gamma\tilde R_3=\tilde A^* 2x(x^2\pa_x F)+x\gamma\tilde R_4
\end{aligned}\end{equation*}
with $\tilde R_3,\tilde R_4$ uniformly bounded in $\Diffscc^1(\Xb)$.

Now,
%
\begin{equation}\begin{split}
&i(\tilde AP-P^*\tilde A)=\sqi [\tilde A,\re P]-(\im P \tilde A+\tilde A\im P)\\
&\qquad=\sqi [A,H-\lambda]+ x^{\rdelta}   R_{18}
-\tilde A^* 2x(x^2\pa_x F)\tilde A-\tilde A 2x(x^2\pa_x F)\tilde A^*
+\gamma R_{21}
  \,.
\end{split}
 \label{15XII17.3}
\end{equation}
%
Thus,
\begin{equation*}\begin{split}
0&=(\psi_F,i(\tilde AP-P^*\tilde A)\psi_F)\\
&\geq (\psi_F,\sqi [A,H]\psi_F)
+2\|x^{1/2}(-x^2\pa_x F)^{1/2}\tilde A\psi_F\|^2+2\|x^{1/2}(-x^2\pa_x
F)^{1/2}\tilde A^*\psi_F\|^2\\
&\qquad- b_{16}\|x^{\rdelta/2}\psi_F\|^2
  -b_{18}\gamma\|\psi_F\|^2.
\end{split}\end{equation*}
Using the Mourre estimate
\eqref{eq:Mourre-0}
in Proposition~\ref{prop:Mourre-est},
we deduce that, with $c'_0\equiv 2\lambda >0$,
\begin{equation*}\begin{split}
 0
  \geq
    c'_0\|\psi_F\|^2&-{b_{19}}
    \|(H-\lambda)\psi_F\|\|\psi_F\|
    -{b_{20}}\|x^{\rdelta/2}\psi_F\|^2
\\
 &
  \qquad\qquad
    -b_{18}\gamma\|\psi_F\|^2
    \,.
\end{split}\end{equation*}
Using \eqref{eq:deg-re-est}-\eqref{eq:deg-im-est} we deduce, as above, that
\begin{equation*}
 \begin{split}
 0
  &
   \geq c'_0\|\psi_F\|^2-\gamma b_{21}\|\psi_F\|^2
    -b_{22}\|\psi_F\|\|x^{\rdelta} \psi_F\|-b_{23}\|x^{\rdelta/2}\psi_F\|^2
\\
 &
  \geq (c'_0-\gamma b_{21}-\ep_1 b_{22})\|\psi_F\|^2
   -  b_{22}\ep_1^{-1} \|x^{\rdelta}\psi_F\|^2
    - b_{23}    \|x^{\rdelta/2}\psi_F\|^2
    \\
 &
  \geq(c'_0-\gamma b_{21}-\ep_1 b_{22}
     -(b_{22}\ep_1^{-1} \tdelta^{\delta}
   +b_{23}\tdelta^{\delta/2}))\|\psi_F\|^2
    -b_{24}(\tdelta)\ep_1^{-1} \|\psi\|^2
     \,.
 \end{split}
\end{equation*}
Again, we fix first $\ep_1>0$ so that $c'_0-\ep_1 b_{22}>0$, then $\gamma_0>0$
so that $c'_0-\gamma_0 b_{21}-\ep_1 b_{22}>0$, finally $\deltat>0$
so that $c'_0-\gamma_0 b_{21}-\ep_1 b_{22}
     -(b_{22}\ep_1^{-1} \tdelta^{\delta}
   +b_{23}\tdelta^{\delta/2})>0$.
Now letting $\beta\to \infty$ gives that $e^{\gamma/x}\psi\in L^2_{\scl}(\Xb)$
for $\gamma<\gamma_0$,
as above.
\qed\medskip\end{proof}

\subsection{Absence of positive eigenvalues -- high energy estimates}\label{sec:unique}

We next prove that faster than exponential
decay of an eigenfunction of $H$ implies that it vanishes. This was
also the approach taken by Froese and Herbst. However, we use a different,
more robust, approach to deal with our much larger error terms.
The proof is based on conjugation by $\exp(\alpha/x)$ and letting
$\alpha\to+\infty$. Correspondingly, we require positive commutator
estimates at high energies. In such a setting first order terms are
irrelevant, i.e.\ $V$ does not play a significant role below.
Indeed, we work ``semiclassically'' (writing
$h=\alpha^{-1}$), and the key fact we use is that the commutator
of the real and imaginary parts of the conjugated Hamiltonian has
the correct sign on its characteristic variety.

We start by recalling (see \cite{Zworski:Semiclassical} for a general
introduction to semiclassical analysis,
\cite{Vasy-Zworski:Semiclassical} for a specific discussion in the
scattering setting)  that semiclassical scattering vector
fields
$V\in \Vsch(\Xb)$ are simply $h$-dependent families of vector fields
of the form $hV'$, $V'\in\CI([0,1)_h;\Vsc(\Xb))$, i.e.\ $h$ times
scattering vector fields smoothly depending on $h$. (Note that one may
simply choose to have bounded, not smooth, families of vector fields
$V'$ in $\Vsc(\Xb)$, analogously to how we define pseudodifferential
operators -- since $h$ is a parameter, differentiation in it is not an
issue.) The corresponding
differential operators, $P\in\Diffsch^m(\Xb)$ are finite sums of up to $m$-fold
products of these, with $\CI(\Xb\times[0,1))$ coefficients. Thus, in
local coordinates, such
an operator $P$ is of the form
$$
\sum_{j+|\alpha|\leq m} a_{j,\alpha}(x,y,h) (h x^2 D_x)^j(h xD_y)^\alpha.
$$
Ellipticity of such an operator in the usual, differential, sense is the statement that
$$
\Big|\sum_{j+|\alpha|=m} a_{j,\alpha}(x,y,h) \xi^j\eta^\alpha\Big|\geq
c(|\xi|+|\eta|)^m,\qquad |\xi|+|\eta|>R
$$
for some $c>0$ and $R>0$. Note that the $h$ factors appearing in front
of the derivatives are regarded as parts of the expression, i.e.\ it
is $hx^2D_x$ that is turned into $\xi$, etc. One defines $\Diffscch$
analogously, by allowing symbolic (rather than just smooth)
coefficients, smoothly depending on $h$. Note that if
$A\in\Diffsch^{s,r}(\Xb)$, $B\in\Diffsch^{s',r'}(\Xb)$ then
$$
[A,B]\in h\Diffsch^{s+s'-1,r+r'-1}(\Xb),
$$
i.e.\ in addition to the gain in the two orders, there is also an
extra $h$ gained; there is a similar statement for $\Diffscch$.

If $\Xb$ is the radial
compactification of $\RR^n$, obtained by ``adding a sphere at
  $r=\infty$'' (cf., e.g., \cite{RBMSpec}; not to be confused with a
  one-point compactification familiar to general relativists), as
discussed in Section~\ref{s25X17.1},  this means that $P$ is of the form
$$
\sum_{|\beta|\leq m} b_\beta(z,h) (h D_z)^\beta,
$$
where $b_\beta$ are classical symbols smoothly depending on $h$.

The semiclassical Sobolev norms $\|.\|_{\Hsch^{s,r}(\Xb)}$, for $s\geq 0$ integer,
are, for $h$-dependent families of functions in $\Hsc^{s,r}(\Xb)$,
supported in a coordinate chart,
$$
\|u \|^2_{\Hsch^{s,r}(\Xb)}=\sum_{j+|\alpha|\leq
  s}\|x^{-r}(hx^2D_x)^j(h x D_y)^\alpha u\|^2_{L^2_\scl(\Xb)},
$$
and in general via a partition of unity. In the case of the radial
compactification of $\RR^n$, to which the general case locally
reduces, this is equivalent to
$$
\|u \|^2_{\Hsch^{s,r}(\Xb)}=\sum_{|\beta|\leq s}\|\langle z\rangle^r(hD_z)^\beta u\|^2_{L^2},
$$
i.e.\ they are like standard weighted Sobolev spaces but with an $h$
appearing in front of each derivative.

Similarly, semiclassical scattering pseudodifferential operators $A\in\Psiscch^{s,r}(\Xb)$
reduce to semiclassical pseudodifferential operators on $\RR^n$
resulting from semiclassical quantizations
$$
Au(z)=(2\pi h)^{-n}\int_{\RR^n\times\RR^n} e^{i(z-z')\cdot\zeta/h}a(z,\zeta,h) u(z')\,dz'\,d\zeta
$$
of symbols satisfying estimates of the
kind
$$
|D_z^\alpha D_\zeta^\beta a(z,\zeta,h)|\leq C(\alpha,\beta)(1+|z|)^{r-|\alpha|}(1+|\zeta|)^{s-|\beta|},
$$
i.e.\ uniform (in $h$) families of scattering symbols. Notice the
factor $1/h$ appearing in the exponent; one could change variables to
$\zeta'=\zeta/h$ in the integral, then $a$ would be evaluated at
$(z,h\zeta',h)$, explaining the appearance of $h$ in front of
derivatives in the differential operator discussion above when $a$ is a polynomial in $\zeta$. The standard
results, in particular elliptic estimates, hold, so if $A$ is
elliptic in the scattering sense in this semiclassical context,
meaning that $|a|$ has a comparable positive lower bound for
$|z|+|\zeta|$ large,
then elliptic estimates
$$
\|u\|_{\Hsch^{k,p}(\Xb)}\leq C_{k,k',p.p'}(\|Au\|_{\Hsch^{k-s,p-r}(\Xb)}+\|u\|_{\Hsch^{k',p'}(\Xb)})
$$
hold.
Again, there is a version when ellipticity only holds in the
differential order sense, namely if $|\zeta|$ is large; this assumes that
$u\in \Hsch^{k',p}(\Xb)$ (i.e.\ $p=p'$) and
then is of
the form
$$
\|u\|_{\Hsch^{k,p}(\Xb)}\leq C_{k,k',p}(\|Au\|_{\Hsch^{k-s,p-r}(\Xb)}+\|u\|_{\Hsch^{k',p}(\Xb)});
$$
cf.\ the discussion at the beginning of Section~\ref{sec:poly}.

\begin{thm}\label{thm:unique}
Let $\lambda\in\Real$. If $H\psi
=\lambda\psi$, $\exp(\alpha/x)\psi\in L^2_{\scl}(\Xbext)$ for all $\alpha$,
then $\psi\equiv 0$.
\end{thm}

\begin{proof}
As in the previous proofs, for clarity of the argument we assume first that $(\red{M},g)$ is complete and $\psi\in L^2_{\scl}(\Xb)$. We will present the, essentially notational, changes arising when $\psi\in L^2_{\scl}(\Xbext)$ at the end of the proof.

Let
$$
 F=F_\alpha=\phi(x)\frac{\alpha}{x}
 \,,
$$
where $\phi\in\Cinf_c(\Real)$ is
supported near $0$, identically $1$ in a smaller neighborhood of $0$,
and let $\psi_F:=e^F\psi$. Then with $h=\alpha^{-1}$, $H_h=h^2H(F)$   and
$P_h:=H_h-h^2\lambda$
are elliptic  semiclassical  differential operators, elliptic in the usual
sense of differentiable order (i.e.\ the lower bound for the absolute
value of the principal symbol holds for $|\zeta|$ large, as discussed
above), and
\begin{equation*}
P_h\psi_h=0,\ \psi_h\equiv \psi_F
 \,,
\end{equation*}
so by elliptic regularity,
\begin{equation}\label{eq:h-ell-reg}
\|\psi_h\|_{x^pH^k_{\scl,h}(\Xb)}\leq C_1\|\psi_h\|_{x^pL^2_{\scl}(\Xb)},
\end{equation}
$C_1$ independent of $h\in(0,1]$ (but depends on $k$ and $p$). In general,
below the $C_j$ denote constants independent of $h\in(0,1]$ (and $\delta>0$).

The key step in the proof of this theorem arises from considering
\begin{equation*}
P_h^*P_h=(\re P_h)^2+(\im P_h)^2+i(\re P_h\im P_h-\im P_h\re P_h)
\end{equation*}
so
\begin{equation}\label{eq:h-comm-8}
0=(\psi_h,P_h^*P_h\psi_h)=\|\re P_h\psi_h\|^2+\|\im P_h\psi_h\|^2+
(\psi_h,\sqi [\re P_h,\im P_h]\psi_h).
\end{equation}
The first two terms on the right-hand side are non-negative, so the
key issue is the positivity of the commutator. More precisely, we
need that there exist operators $R_j$ bounded in
$\Diff_{\scl,h}^{2,0}(\Xb)$
such that
\begin{equation}\begin{split}\label{eq:h-comm-16}
(\psi_h&,\sqi [\re P_h,\im P_h]\psi_h)\\
&\geq (\psi_h,
(xh+\re P_h xhR_1+xhR_2\re P_h\\
&\qquad\qquad+\im P_h xhR_3+xhR_4\im P_h+xh^2R_5+x^{1+\delta}hR_6)\psi_h).
\end{split}\end{equation}
The important
point is that when ignoring both $\re P_h$ and $\im P_h$ in the right-hand side, the
commutator is estimated from below by a positive multiple of $xh$,
plus terms $O(xh^2)$ and $O(x^{1+\delta}h)$.

We first prove \eqref{eq:h-comm-16}, and then show how to use it to
prove the theorem. First, modulo terms that will give contributions that
are in the error terms, $\re P_h$ may be replaced by $h^2\Delta_g-1 -
h^2 \lambda$, and indeed $h^2\Delta_g-1$,
while $\im P_h$ may be replaced by $-2h(x^2D_x)$ (compare Lemma~\ref{lemma:FH-1}). Now, by a principal
symbol calculation, see
Proposition~\ref{prop:Mourre-est},
\begin{equation*}\begin{aligned}
 \sqi [h^2\Delta_{g_0} -1,-2h(x^2D_x)]&=\sqi [h^2\Delta_{g_0},-2h(x^2D_x)]\\
&=xh(4h^2\Delta_{g_0}-4h^2(x^2 D_x)^2+R_7),
\ R_7\in x^\rdelta\Diff^2_{\scl,h}(\Xb)\\
&=xh(4+4\re P_h-4\im P_h^2+R_7'), \ R'_7\in x^\rdelta\Diff^2_{\scl,h}(\Xb)\\
\end{aligned}\end{equation*}
which is of the desired form.

We now show how to use \eqref{eq:h-comm-16} to show unique
continuation at infinity. Let $x_0=\sup_{\Xb} x$.
We first remark that
\begin{equation*}\begin{split}
&|(\psi_h, xh R_2\re P_h\psi_h)|\leq C_2 h\|x  \psi _h\|\|\re P_h\psi_h\|
\leq C_2 h\|x  \psi _h\|^2+C_2 h\|\re P_h\psi_h\|^2,\\
&|(\psi_h, xh R_4\im P_h\psi_h)|\leq C_3 h\|x  \psi _h\|\|\im P_h\psi_h\|
\leq C_3 h\|x  \psi _h\|^2+C_3 h\|\im P_h\psi_h\|^2,
\end{split}\end{equation*}
with similar expressions for the $R_1$ and $R_3$ terms in
\eqref{eq:h-comm-16}. Next,
\begin{equation*}\begin{split}
 &|(\psi_h,xh^2R_5\psi_h)|\leq C_4 h^2\|x^{1/2}\psi_h\|^2
 \,,
\\
 &|(\psi_h,x^{1+\delta}h R_6\psi_h)|\leq C_5 h\|x^{(1+\delta)/2}\psi_h\|^2
  \,.
\end{split}\end{equation*}
For $\deltat>0$, in $x\geq\deltat$, $|\psi_h|=e^{1/{xh}}|\psi|\leq
e^{1/(\deltat h)}|\psi|$, so
\begin{equation*}\begin{split}
\|x^{(1+\delta)/2}\psi_h\|^2&=\|x^{(1+\delta)/2}\psi_h\|^2_{x\leq\deltat}+\|x^{(1+\delta)/2}\psi_h\|^2_{x\geq\deltat}\\
&\leq \deltat^\delta\|x^{1/2}\psi_h\|^2_{x\leq\deltat}
+x_0^{1+\delta} e^{2/(\deltat h)}\|\psi\|^2_{x\geq\deltat}\\
&\leq \deltat^\delta\|x^{1/2}\psi_h\|^2
+x_0^{1+\delta} e^{2/(\deltat h)}\|\psi\|^2.
\end{split}\end{equation*}
Thus,
\begin{equation*}
\|(\psi_h,x^{1+\delta}h R_6\psi_h)|\leq C_5 h\deltat^\delta\|x^{1/2}\psi_h\|^2
+C_5 x_0^{1+\delta} he^{2/(\deltat h)}\|\psi\|^2.
\end{equation*}
Hence, we deduce from \eqref{eq:h-comm-8}-\eqref{eq:h-comm-16} that
\begin{equation}
 \begin{split}
  0\geq & (1-C_6 h)\|\re P_h\psi_h\|^2
  +(1-C_7 h)\|\im P_h\psi_h\|^2
\\
&+h(1-C_8 h-C_9\deltat^\delta)\|x^{1/2}\psi_h\|^2
-C_{10}he^{2/(\deltat h)}\|\psi\|^2.
\end{split}
 \label{23XII17.1+}
\end{equation}
Hence, there exists $h_0>0$ such that for $h\in(0,h_0)$,
\begin{equation}\label{eq:h-comm-64}
C_{10} h e^{2/(\deltat h)}\|\psi\|^2
\geq h(\frac{1}{2}-C_9\deltat^\delta)\|x^{1/2}\psi_h\|^2
\end{equation}

Now suppose that $\deltat\in \big(0,\min(\big(\frac{1}{4C_9}\big)^{1/\delta},\frac{1}{h_0})\big)$ and
$\supp\psi\cap\{
x\leq\frac{\deltat}{4}\}$ is non-empty.
Since $xe^{2/xh}=h^{-1}f(xh)$ where $f(t)=te^{2/t}$, and $f$ is decreasing
on $(0,2)$ (its minimum on $(0,\infty)$ is assumed at $2$), we deduce
that for $x\leq\deltat/2$, $x e^{2/xh}\geq \frac{\deltat}{2} e^{4/(\deltat h)}$,
so
\begin{equation*}
\|x^{1/2}\psi_h\|^2\geq C_{11} \deltat e^{4/(\deltat h)},\ C_{11}>0.
\end{equation*}
Thus, we conclude from \eqref{eq:h-comm-64} that
\begin{equation}
 \label{23XII17.2}
C_{10}\|\psi\|^2\geq (\frac{1}{2}-C_9\deltat^\delta)C_{11} \deltat
e^{2/(\deltat h)}.
\end{equation}
But letting $h\to 0$,
the right-hand side goes to $+\infty$, providing a contradiction.

Thus, $\psi$ vanishes for $x\leq \deltat/4$, hence vanishes
identically on $\Xb$ by the usual Carleman-type unique continuation theorem
\cite[Theorem~17.2.1]{Hor}, and when $\psi\in L^2_{\scl}(\Xb)$ we are done.

As already hinted at, the result for $\psi\in L^2_{\scl}(\Xbext)$ follows by change of notation. Namely, let now $\tilde \psi_h$ denote the original eigenfunction, thus we have $P_h\tilde \psi_h=0$. Let $\psi_h=\chi\tilde\psi_h$, where $\chi\equiv 1$  on a neighborhood of $\supp\phi$, i.e.\ $\supp F$, but still supported in a collar neighborhood of the Euclidean end. We have
$$
P_h\psi_h=\chi P_h\tilde\psi_h+[P_h,\chi]\tilde\psi_h=[P_h,\chi]\tilde\psi_h,
$$
and by construction, namely on $\supp d\chi$ where the weight $F$ vanishes, we have
$$
P_h\psi_h=[P_h,\chi]\tilde\psi
 \,.
$$
In particular all semiclassical Sobolev norms of $P_h\psi_h$ are bounded by a constant $C$ (independent of $h\in(0,1]$, of course depending on the norm).

We repeat the argument above. The left hand side of  \eqref{eq:h-comm-8}
is not zero anymore, rather $\|P_h\psi_h\|^2$, which is bounded by $C^2$. The rest of the computation is unchanged until \eqref{23XII17.1+}, where the left-hand side becomes $C^2$ instead of $0$. Hence, \eqref{eq:h-comm-64} also has a $C^2$ added to the left-hand side, and then
\eqref{23XII17.2} becomes
\begin{equation}
 \label{23XII17.3}
C_{10}\|\psi\|^2+C^2 h^{-1}e^{-2\tdelta/h}\geq (\frac{1}{2}-C_9\deltat^\delta)C_{11} \deltat
e^{2/(\deltat h)}.
\end{equation}
Since the new term also goes to $0$ as $h\to 0$, the final step of the argument is unchanged, whence $\psi$ vanishes for small $x$.

We conclude again that our original eigenfunction $\tilde\psi$ vanishes for small $x$, and the usual elliptic unique continuation result finishes the proof.
\qed\medskip\end{proof}

\bigskip

\noindent{\sc Acknowledgements:}   PTC was supported in part by  the Austrian Science Fund (FWF) under project P 23719-N16, and by
Narodowe Centrum Nauki under the grant DEC-2011/03/B/ST1/02625.
AV thanks the NSF for partial support under grant number DMS-1664683, and the Simons Foundation for partial support via a Simons fellowship grant.
  The authors are grateful to the Erwin Schr\"odinger Institute, Vienna, for hospitality and support during part of work on this paper.

\providecommand{\bysame}{\leavevmode\hbox to3em{\hrulefill}\thinspace}
\providecommand{\MR}{\relax\ifhmode\unskip\space\fi MR }
\providecommand{\MRhref}[2]{%
  \href{http://www.ams.org/mathscinet-getitem?mr=#1}{#2}
}
\providecommand{\href}[2]{#2}

%
%

\end{document}